\documentclass[reqno, 11pt, a4paper]{amsart}

\usepackage{latexsym,ifthen,xspace}
\usepackage{amsmath,amssymb,amsthm}
\usepackage{bbm}
\usepackage{enumerate, calc}
\usepackage[colorlinks=true, pdfstartview=FitV, linkcolor=blue,
citecolor=blue, urlcolor=blue, pagebackref=false]{hyperref}
\usepackage[text={33pc,605pt},centering]{geometry}    %11pt
\usepackage{graphicx}

\usepackage{tikz}

\newtheorem{theorem}{Theorem}[section]
\newtheorem{lemma}[theorem]{Lemma}
\newtheorem{prop}[theorem]{Proposition}

\newtheorem{corro}[theorem]{Corollary}

\theoremstyle{definition}
\newtheorem{definition}[theorem]{Definition}
\newtheorem*{definition*}{Definition}
\newtheorem{example}[theorem]{Example}

\theoremstyle{remark}
\newtheorem{remark}[theorem]{Remark}

\numberwithin{equation}{section}

\DeclareMathAlphabet{\mathsl}{OT1}{cmss}{m}{sl}
\SetMathAlphabet{\mathsl}{bold}{OT1}{cmss}{bx}{sl}

\newcommand{\al}{\ensuremath{\alpha}}

\newcommand{\ga}{\ensuremath{\gamma}}
\newcommand{\de}{\ensuremath{\delta}}

\newcommand{\ze}{\ensuremath{\zeta}}

\newcommand{\la}{\ensuremath{\lambda}}

\newcommand{\si}{\ensuremath{\sigma}}

\newcommand{\ve}{\ensuremath{\varepsilon}}

\newcommand{\vr}{\ensuremath{\varrho}}

\newcommand{\vp}{\ensuremath{\varphi}}

\newcommand{\De}{\ensuremath{\Delta}}
\newcommand{\Th}{\ensuremath{\Theta}}

\newcommand{\Om}{\ensuremath{\Omega}}
\newcommand{\cD}{\ensuremath{\mathcal D}}

\newcommand{\cF}{\ensuremath{\mathcal F}}

\newcommand{\cL}{\ensuremath{\mathcal L}}
\newcommand{\cM}{\ensuremath{\mathcal M}}

\newcommand{\bbN}{\ensuremath{\mathbb N}}

\newcommand{\bbR}{\ensuremath{\mathbb R}}

\newcommand{\bbY}{\ensuremath{\mathbb Y}} 
 
\newcommand{\bfE}{\ensuremath{\mathbf E}}

\DeclareMathOperator{\mean}{\mathbb{E}}

\DeclareMathOperator{\prob}{\mathbb{P}}

\DeclareMathOperator{\tr}{\mathrm{tr}}
\DeclareMathOperator{\rk}{\mathrm{rk}}
\DeclareMathOperator{\diag}{\mathrm{diag}}

\newcommand{\ldef}{\ensuremath{\mathrel{\mathop:}=}}

\newcommand{\mAct}{m_{\textnormal a}}
\newcommand{\mDor}{m_{\textnormal d}}

\newcommand{\indicator}{\ensuremath{\mathbbm{1}}}

\begin{document}

\title[Dormancy and fluctuating environments]{A branching process model for dormancy and seed banks in randomly fluctuating environments}

%    Remove any unused author tags.

%    author one information
\author{Jochen Blath}
\address{Technische Universit\"at Berlin}
\curraddr{Strasse des 17. Juni 136, 10623 Berlin}
\email{blath@math.tu-berlin.de}
\thanks{}

%    author two information
%\author{Matthias Hammer}
%\address{Technische Universit\"at Berlin}
%\curraddr{Strasse des 17. Juni 136, 10623 Berlin}
%\email{hammer@math.tu-berlin.de}
%\thanks{}

%    author three information
\author{Felix Hermann}
\address{Technische Universit\"at Berlin}
\curraddr{Strasse des 17. Juni 136, 10623 Berlin}
\email{hermann@math.tu-berlin.de}
\thanks{}

%    author four information
\author{Martin Slowik}
\address{Technische Universit\"at Berlin}
\curraddr{Strasse des 17. Juni 136, 10623 Berlin}
\email{slowik@math.tu-berlin.de}
\thanks{}

%\subjclass[2000]{}

\keywords{Dormancy; persistence; Bienaym\'e-Galton-Watson process; branching process; seed bank; fitness}

\date{\today}

%\dedicatory{}

\begin{abstract}
  The goal of this article is to contribute towards the conceptual and quantitative understanding of the evolutionary benefits for (microbial) populations to maintain a seed bank (consisting of dormant individuals) when facing fluctuating environmental conditions. To this end, we compare the long term behaviour of `1-type' Bienaym\'e-Galton-Watson branching processes (describing populations consisting of `active' individuals only) with that of a class of `2-type' branching processes, describing populations consisting of `active' and `dormant' individuals. All processes are embedded in an environment changing randomly between `harsh' and `healthy' conditions, affecting the reproductive behaviour of the populations accordingly. For the 2-type branching processes, we consider several different switching regimes between active and dormant states. We also impose overall resource limitations which incorporate the potentially different `production costs' of active and dormant offspring, leading to the notion of `fair comparison' between different populations, and allow for a reproductive trade-off due to the maintenance of the  dormancy trait. Our switching regimes include the case where switches from active to dormant states and vice versa happen randomly, irrespective of the state of the environment (`spontaneous switching'), but also the case where switches are triggered by the environment (`responsive switching'), as well as combined strategies. It turns out that there are rather natural scenarios under which either switching strategy can be super-critical, while the others, as well as complete absence of a seed bank, are strictly sub-critical, even under `fair comparison' wrt.\ available resources. In such a case, we see a clear selective advantage of the super-critical strategy, which is retained even under the presence of a (potentially small) reproductive trade-off.  Mathematically, our results rest on the control of Lyapunov exponents related to random matrix products (governed by the dynamics of the environment and the switching regimes).  While their exact computation in general is considered a `notoriously difficult problem', we provide some insight into the structure of switching regimes related to dormancy that allow us to achieve this control at least in important special cases.  Our rigorous results extend and complement earlier theoretical work on dormancy (and the related theory of `phenotypic switches'), including Kussel and Leibler (2005) and Malik and Smith (2008) in the deterministic set-up, and the multi-type branching process model by Dombry, Mazza and Bansaye (2011) in the stochastic set-up.
\end{abstract}

\maketitle

% \tableofcontents

\section{Introduction}
\label{sec:intro}
\subsection{Biological motivation}
Dormancy is an evolutionary trait that comes in many guises and has evolved independently multiple times across the tree of life.  In particular, it is ubiquitous among microbial communities.  As a general definition, we say that an individual exhibits a \emph{dormancy trait} if it is able to enter a reversible state of vanishing metabolic activity.  It has been observed that a large fraction of microbes on earth is currently in a dormant state, thus creating \emph{seed banks} consisting of inactive individuals (see e.g.\ Lennon and Jones \cite{LJ11} and Shoemaker and Lennon \cite{SL18} for recent overviews).  A common ecological and evolutionary explanation for the emergence of the corresponding complex dormancy traits is that the maintenance of a seed bank serves as a bet-hedging strategy to ensure survival in fluctuating and potentially unfavourable environmental conditions.  Recent theory has also shown that dormancy traits can already be beneficial in competing species models in the presence of sufficiently strong competitive pressure for limited resources (even in otherwise stable environments), see \cite{BT20}.  However, maintaining a dormancy trait is costly and comes with a substantial trade-off: For example, microbes need to invest resources into resting structures and the machinery required for switching into and out of a dormant state, which are then unavailable for reproduction. 

Dormancy also has implications for the pathogenic character of microbial communities and plays an important role in human health.  For example, dormancy in the form of persister cells can lead to chronic infections, since these cells can withstand antibiotic treatment (\cite{BMCKL04, FGH17, Le10}).  Further, dormancy, both on the level of individual cells as well as the tumor level, plays a crucial role in cancer dynamics \cite{EI19}.  In all of the above situations external treatment can be seen as a form of environmental stress for the pathogens.

Hence, improving the conceptual and quantitative understanding of the mechanisms leading to fitness advantages for individuals with a dormancy trait in fluctuating environments, incorporating the potentially different costs of forming active and dormant offspring (and potential reproductive trade-offs due to the maintenance of dormancy traits), seems to be a worthwhile task. 

\subsection{Deterministic vs.\ stochastic population dynamic modeling and known results}
In mathematical population dynamics, there is a classical dichotomy between deterministic and stochastic modeling, and both approaches have been employed extensively and successfully in the past.  Unsurprisingly, this traditional distinction is also present in the rather recent literature on populations exhibiting dormancy (or more the more general concept of `phenotypic switches').  We now briefly review some important approaches and results in both areas.

In the last two decades, dormancy-related population dynamic modeling based on \emph{deterministic dynamical systems} has expanded rather rapidly, often with a focus on phenotypic plasticity in microbial communities, see e.g.\ \cite{BHMP02, BMCKL04, KKBL05, KL05, MS08, FW18}.  The important paper by Balaban et.\ al.~\cite{BMCKL04} for example describes `persistence' (which can be seen as a form of dormancy) as a phenotypic switch, and several of the above papers deal with models incorporating various switching strategies and fluctuating environments.  Kussel et.\ al.~\cite{KKBL05} consider periodic antibiotic treatment (and also treat a stochastic version of their model via simulation), and Kussel and Leibler \cite{KL05} incorporate randomly changing environments, however under the condition that the random changes are slow.  Fitness is typically measured in terms of the (maximal) Lyapunov exponents of the underlying dynamical systems, which is often difficult to evaluate analytically in the presence of random environments.  Kussel and Leibler approximate the Lyapunov exponent under a `slow environment condition', reducing the model to an essentially one-dimensional system, which is a strategy that we will meet again in different forms in the sequel.  These models have been taken up in a mathematical article by Malik and Smith \cite{MS08}, which provide a set of rigorous results regarding the maximal Lyapunov exponents of dynamical systems explicitly incorporating dormancy, considering both stochastic and responsive switches in (periodically) changing environments.  They also compare these to the corresponding results for populations without dormancy trait (so-called `sleepless population').  However, for truly random environments, they do not provide exact representations for the maximal Lyapunov exponents and instead give relatively simple (yet useful) bounds.

Recently, also \emph{stochastic (individual-based) models} for seed banks and dormancy have gained increasing attention, in particular in \emph{population genetics} (\cite{KKL01, BEGCKW15, BGCKW16, BGKW18}).  However, these models are mainly concerned with genealogical implications of seed banks and typically require constant population size (without random environment).  In \emph{population dynamics}, while there are interesting recent simulation studies such as \cite{LFL17}, rigorous mathematical modeling and results are still relatively rare.  Here, a suitable framework for individual-based seed bank models with fluctuating population size is given by the theory of multi-type branching processes (in random environments).  Indeed, dormancy has been described in a brief example in the book by Haccou, Jagers and Vatutin \cite[Example 5.3]{HJV07} as a 2-type branching process, which served as a motivation for this paper.  For \emph{quiescence} (which is a similar concept as dormancy), a multi-type branching-process model has been proposed in \cite{AJ11}, including a simulation study.  On the theoretical side, in the context of \emph{phenotypic switches}, Dombry, Mazza and Bansaye \cite{DMB11} and Jost and Wang \cite{JW14}, again building on motivation from \cite{KL05}, have developed a branching-process based framework for phenotypic plasticity and obtained interesting rigorous results on the optimality of switching strategies in random environments.  However, their set-ups and results, though closely related, do not focus on dormancy, and only partially cover the reproductive and switching strategies that we are going to discuss below (we will explicitly comment on the differences wrt.\ our model and results in the sequel).  They are able to determine Lyapunov exponents in random environments for their model, but again under a condition which essentially restricts the problem to a one-dimensional system.  Finally, regarding {phenotypic switches} specifically related to cancer biology, interesting models and results, in an interdisciplinary framework, have been provided in \cite{B+16, BB18, G+19}.  These papers contain mathematical and simulation based results in the very concrete situation of immune-therapy of melanoma, where cancer cells exhibit phenotypic plasticity.  However, the modeling approaches do not cover random environments.  Yet they show the power and need for stochastic individual-based modeling in such situations.

\subsection{Modeling approach and organization of the present paper}
Our approach, motivated by the example in \cite{HJV07}, is based on a 2-type branching process $(Z_n)=(Z^1_n, Z^2_n)$ with $Z^1_n$ denoting the active and $Z^2_n$ denoting the dormant individuals at time/generation $n$, which we embed in a random environment that is described by a stochastic process $(E_n)$ and governs the respective sequence of random reproductive laws $(Q(E_n))$.  As in \cite{KL05, MS08, DMB11}, we will discuss results related to both stochastic/spontaneous as well as responsive switching strategies.  Further, we also consider mixed strategies.  We aim for explicit results under `fair comparison' regarding resource limitations including potentially different costs for active and dormant offspring, and also in comparison with a 1-type branching process without dormancy trait (`sleepless case'), expressing qualitative and quantitative \emph{fitness advantages} in terms of the maximal Lyapunov exponent.

By providing a model tailored to dormancy in a random environment, we close a gap related to the multi-type branching process models and results provided in \cite{DMB11} and \cite{JW14} related to phenotypic switching, which only partially cover our dormancy models and results, and extend and refine results of \cite{MS08} which explicitly model dormancy in the deterministic dynamical systems case, but with a smaller set of switching strategies and few results for truly random environments.  Additionally, we pay particular attention to reproductive costs related to dormancy.

All models and results will be introduced and discussed in Section~\ref{sec:models}.  We observe that there are natural parameter regimes under which either the spontaneous or the responsive switching strategies, or even a mixture of strategies, will be fit, while all the others are detrimental.  We discuss the corresponding parameter regimes in detail and visualize them in certain important cases, see e.g.\ Figures~\ref{fig:phase_diagram}, \ref{fig:phase_diagram_cc} and \ref{fig:phase_diagram:fair} below.  This shows that already in our relatively simple random environment (involving only two states), dormancy leads to a rather rich picture regarding the long-term behaviour of the embedded branching processes. 

However, while our results will capture several prototypical scenarios corresponding to both stochastic/spontaneous and responsive switching (and mixtures), we are still far from being able to provide a mathematically complete classification in the full space of switching strategies.  One theoretical reason for this is that computing the maximal Lyapunov exponent of a random multiplicative sequence of positive matrices, which is the mathematical core of the problem, is infeasible in general (see e.g.\ \cite{BL85, Le82} for an overview of the mathematical theory), and works only if the underlying matrices exhibit additional algebraic properties.

Hence, a further aim of this paper is to provide a small review of current methods to compute / estimate maximal Lyapunov exponents.  The reason which allows \cite{DMB11} and \cite{JW14} to treat the responsive switching regime in their papers is that their assumptions reduce the system to an essentially one-dimensional case, which interestingly has a similar effect as the `slow variation assumption' of Kussel and Leibler \cite{KKBL05}, and we will investigate similar cases.  While the exponents are easily accessible for the corresponding `rank-1 matrices' (and, of course, scalars), spontaneous/stochastic switching strategies can a priori involve both `rank-1' and `rank-2 matrices'.  We show that while special cases of the stochastic switching regime can again be treated with the rank-1 approach, for the general stochastic switching regime involving rank-2 matrices, techniques used in \cite{MS08} are available, which lead at least to bounds on the Lyapunov exponents.  We also provide further bounds and estimators.  These theoretical considerations can be found in Section~\ref{sec:theory}.

Finally, in Section~\ref{sec:discussion}, we discuss some open questions and potential further steps in modeling and analysis of dormancy and seed banks in random environments, from a somewhat theoretical perspective.

\section{Models and main results}
\label{sec:models}
Recall that for a classical (1-type) Bienaym\'e-Galton-Watson process, say $X=(X_n)$, it is assumed that individuals die and reproduce independently of each other according to some given common offspring distribution $Q_X$ on $\bbN_0$.  We extend this framework by introducing a second component that acts as a reservoir of \emph{dormant} individuals, also referred to as \emph{seed bank}.  Moreover, we allow the offspring distribution in each generation $n$ to depend on the state of a random environment process $E=(E_n)$.  This gives rise to a particular class of 2-type Bienaym\'e-Galton-Watson processes in random environment that we introduce formally in Definition~\ref{def:bgwdre} and which will be the main object of study in this paper. However, in the sequel, we will also discuss more general $p$-type branching processes (for $p \ge 1$) in random environments, so that we will first introduce the corresponding general notation, which is standard in the theory of multi-type branching processes.

\subsection*{Notation}
Let $E=(E_n)_{n \in \bbN_0}$ be a stationary and ergodic Markov chain on some probability space $(\Om, \cF, \prob)$ taking values in some measurable space $(\Om', \cF')$ and denote by $\pi_E$ its stationary distribution.  Such a process will be called \emph{random environment process}.  For $p \in \bbN$, we write $\cM_1(\bbN_0^p)$ to denote the space of probability measures on $\bbN_0^p$, and set $\Th \ldef \{(Q^{1}, \ldots, Q^{p}) : Q^{i} \in \cM_1(\bbN_0^p)\}$.  Elements of $\Th$ will be interpreted as the collection of the $p$ \emph{offspring distributions} on $\bbN_0^p$ (one for each type).  An infinite sequence $\Pi = (Q(E_1), Q(E_2), \ldots)$ generated by $(E_n)$ and a random variable $Q\!: \Om' \to \Th$ will be called sequence of \emph{random offspring distributions} with respect to the {environment process} $(E_n)$.  Finally, a sequence of $\bbN_0^p$-valued random variables $Z_0, Z_1, \ldots$ will be called a \emph{$p$-type Bienaym\'e-Galton-Watson process in random environment} ($p$-type BGWPRE), if $Z_0$ is independent of $\Pi$ and if for each given realization $(e_1, e_2, \dots)$ of $E$ (and thus also of $\Pi$) the process $Z = (Z_n)_{n \in \mathbb{N}_0}$ is a Markov chain whose law satisfies
\begin{align*}
  \cL\big(
    Z_{n} \mid Z_{n-1} = z,\, \Pi = (Q(e_1), Q(e_2), \ldots)
  \big)
  \;=\;
  \cL\bigg(
    \sum_{i=1}^p \sum_{j=1}^{z^i}\ze_{j}^{i}
  \bigg),
\end{align*}
for every $n \in \bbN$ and $z=(z^1, \dots, z^p) \in \bbN_0^p$, where the $(\ze_j^i : i \in \{1, \ldots, p\}, j \in \bbN)$ are independent random variables taking values in $\bbN_0^p$, and for each $i \in \{1, \ldots, p\}$, the $\ze_1^i, \ze_2^i, \ldots$ are identically distributed according to $Q^i(e_n)$.  In the language of branching processes, if the state of the environment at time $n$ is $e_n \in \Om'$, then each of the $Z_n^i$ individuals of type $i$ alive at time $n$ produces offspring according to the probability distribution $Q^i(e_n)$, independent of the offspring production of all the other individuals.  For notational clarity, we will often write $Q_Z$ to denote the random variable $Q$ that is used in the definition of a branching process $Z$. 

We are now ready to define the class of branching processes modeling \emph{dormancy}: 
\begin{definition}\label{def:bgwdre}
  With the above notation (for $p=2$), a 2-type BGWPRE $Z=(Z_n)$ will be called a \emph{Bienaym\'e-Galton-Watson process with dormancy in random environment} $(E_n)$, abbreviated BGWPDRE, if, $\prob$-almost surely,
  \begin{align}\label{eq:def:BGWDRE}
    Q_Z^2(E_n)[\{(0,0),(1,0),(0,1)\}]
    \;=\;
    1
    \qquad \forall\, n \in \mathbb{N}_0.
  \end{align}
  Particles of type 1 are called \emph{active} and particles of type 2 are called \emph{dormant}. 
\end{definition}
Note that Condition \eqref{eq:def:BGWDRE} ensures that a (dormant) type 2 particle can either switch its state to type 1 (active), remain in the dormant state 2, or die --  no other transitions are possible. There is no restriction on the offspring reproduction of active type 1 particles other than the following first moment condition.

Throughout we assume, for $\prob$-a.e.\ realization $(e_1, e_2, \ldots)$ of $(E_n)$ and any $n \in \mathbb{N}$,  that the distribution $Q(e_n) \in \Th$ is such that the corresponding random variables $\ze^i=(\ze^1, \ldots, \ze^p)$ distributed according to $Q(e_n)$ satisfy $\mean[\ze^i] < \infty$ for all $i \in \{1, \ldots, p\}$.  Moreover, we write
\begin{align*}
  m_n^{i,j}
  \;\equiv\;
  m^{i,j}(e_n)
  \;\ldef\;
  \mean\!\big[Z_{n+1}^j \mid Z_n = (\delta_{ik})_k,\, \Pi_n = Q_Z(e_n) \big].
\end{align*}
to denote the expected number of offspring of type $j$ produced by a single particle of type $i$ in generation $n$ in the environment $Q_Z(e_n)$, and we denote by 
\begin{align*}
  M_n \;\equiv\; M(e_n) \;\ldef\; (m_n^{i,j})_{i,j}
\end{align*}
the corresponding mean matrix.  Suppose that, for any $n \in \bbN$, the matrix $M_n$ is irreducible.  Then, by the  Perron-Frobenius-Theorem, the spectral radius $\vr_n \equiv \vr(M_n)$ of $M_n$ is a simple eigenvalue with $|\la| \leq \vr_n$ for any (other) eigenvalue $\la$ of $M_n$.

\subsection{Branching processes with dormancy in constant environment}\label{sec:mot-exp}
As a gentle warm-up, we first compare the survival probabilities and extinction times of a classical 1-type BGWP with those of a 2-type BGWPDRE in the absence of environmental fluctuations (in this case, we use the abbreviation BGWPD).  For simplicity, we further restrict ourselves to the binary branching case (following the set-up of Example 5.3 in \cite{HJV07}), which can be thought of as a model for bacterial reproduction via binary fission resp.\ sporulation as exhibited e.g.\ by \emph{Bacillus subtilis}, and summarize several standard results (that nevertheless will be proved in the appendix for the reader's convenience). These results will then serve as a motivation and reference point for our later results involving fluctuating environments, which will also deal with more general reproductive mechanisms.

Let $p \in (0,1)$, $X_0 = 1$ and $X = (X_n)_{n \in \bbN_0}$ be a 1-type BGWP with offspring distribution $Q_X = p \de_2 + (1-p)\de_0$, where $\de$ denotes the Dirac measure. This mechanism can be seen as a caricature of reproduction via cell division:  Every individual in each generation independently either splits in two (cell division) with probability $p$ or dies with probability $1-p$.

Furthermore, for $\ve \in (0,p)$, $b,w \in (0,1)$ and $d \in (0, 1-w)$, let $Z_0 = (1,0)$ and $Z = (Z_n)_{n \in \bbN_0}$ be a 2-type BGWPD with offspring distribution, $Q_Z$, given by
\begin{align*}
  &Q_Z^1(0,0) \;=\; 1-p + \ve,&
  &Q_Z^2(1,0) \;=\; w,\\
  &Q_Z^1(2,0) \;=\; (p-\ve) b,&
  &Q_Z^2(0,0) \;=\; d,\\
  &Q_Z^1(0,1) \;=\; (p-\ve)(1-b),&
  &Q_Z^2(0,1) \;=\; 1-w-d.
\end{align*}
\begin{figure}
  \includegraphics[width=12cm]{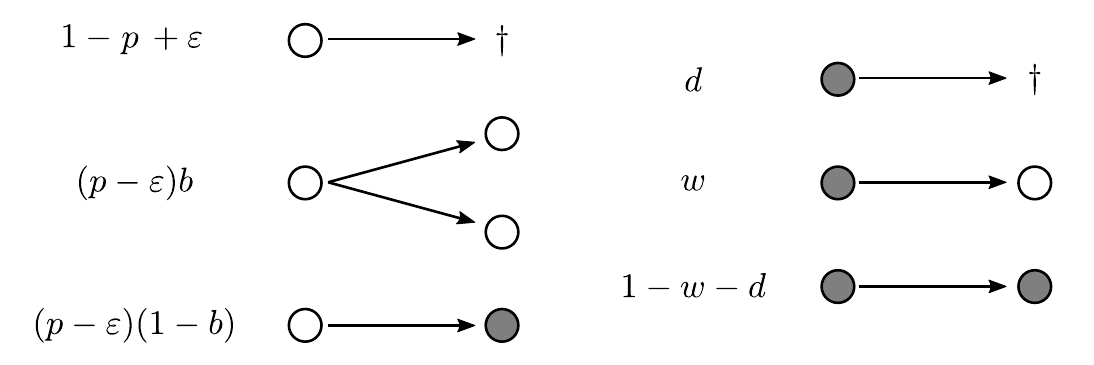}
  \caption{\label{fig:seedbank-penalty-model}
    Offspring distribution of $Z$ for active (white) individuals on the left and dormant (gray) individuals on the right.
  }
\end{figure}
Figure~\ref{fig:seedbank-penalty-model} illustrates the model.  The parameters can be interpreted as follows: $p-\ve > 0$ is the probability with which an active individual either exhibits a reproductive or a switching event. In this case, a binary split will happen with probability $b$ (reproduction), and a switching event into a dormant state (e.g.\ by sporulation) with probability $1-b$. Note that  a `switch' can be thought of as the simultaneous death of an active individual and the corresponding birth of a dormant individual. With probability $q+\ve$, an active individual will die. Dormant individuals resuscitate (``wake up'') with probability $w$, and die with probability $d$, otherwise they stay in their dormant state (with probability $1-w-d$). Note that for $\ve = 0$ and $b=1$, the active component $(Z^1_n)$ of $(Z_n)$ equals $(X_n)$ in distribution.  Hence, $\ve$ can be seen as a way of incorporating a reproductive trade-off that arises from the maintenance costs of the dormancy trait, delivering a reduced splitting (and hence increased death) probability in comparison to the 1-type model. Additionally, note that in our model the potential to switch into dormancy also reduces the reproductive capability, since entering the seed bank is only possible during a `reproduction-or-switching event' at a chance of $1-b$. 

From the following more general result, which we will prove in the Appendix, we obtain a comparison of long-term survival behaviour of $X$ and $Z$.
\begin{prop}\label{prop:motivating-exp}
  Let $X = (X_n)$ be a 1-type BGWP, and $Z = (Z_n)$ a BGWPD with $X_0 = 1$ and $Z_0 = (1,0)$.  Assume that the offspring distributions $Q_X$ and $Q_Z$, respectively, are of finite variance with $Q^2_Z(0,0) > 0$ and $\prob[Z^2_1>0] > 0$.  Set 
  \begin{align*}
    \mu_X \;=\; \mean[X_1]
    \qquad \text{and} \qquad
    \mu_{Z,1} \;=\; \mean[Z_1^1+Z_1^2]
  \end{align*}
  and denote by
  \begin{align*}
    &\si_{Z} \;\ldef\; \prob\!\Big[\lim_{n \to \infty} Z^1_n+Z^2_n > 0 \Big],
    &&& 
    &\si_{X} \;\ldef\; \prob\!\Big[\lim_{n \to \infty} X_n>0 \Big],
    \\[1ex]
    &
    T_{Z} \;\ldef\; \inf\big\{n \geq 1 \;:\; Z_n = 0 \big\}
    &\text{and}&&
    &T_{X} \;\ldef\; \inf\big\{n \geq 1 \;:\; X_n = 0 \big\}
  \end{align*}
  the survival probabilities and extinction times of $Z$ and $X$, respectively.  If
  \begin{align}
    Q_X(k)
    \;\geq\;
    \sum_{\ell=0}^k Q_Z(k-\ell,\ell)
    \label{eq:fair-comp-mot-exp}
  \end{align}
  for all $k \geq 1$, then the following statements hold:
  \begin{enumerate}
  \item If $\mu_X>1$, then $\si_{Z} < \si_{X}$.
    \\[-.5em]
  \item If $\mu_X=1$, then $\si_{Z} = \si_{X}=0$ and $\, \mean\!\big[T_{Z}\big] < \mean\!\big[T_{X}\big] = \infty$.
    \\[-.5em]
  \item If $\mu_X<1$, then $\si_{Z} = \si_{X} = 0$. However, $Q^2_Z$ can be chosen in such a way that $Q^2_Z(1,0)>0$, $Q^2_Z(0,0)>0$ and for some $n_0\in\bbN$
    \begin{align}  
      \prob\!\big[T_{Z}>n \big] &\;>\; \prob\!\big[T_{X} > n\big]
      &&\text{for all } n \geq n_0.
      \label{eq:mot-ex-tails}
    \end{align}
  \end{enumerate}
\end{prop}
Condition \eqref{eq:fair-comp-mot-exp} ensures that the total amount of offspring of active individuals in $Z$ is stochastically dominated by the amount of offspring in $X$.

Proposition~\ref{prop:motivating-exp} shows that -- at least in the simple binary model -- in the super-critical case $p>1/2$ (i.e. $\mu_X=2p>1$) maintaining a seed bank always leads to a decreased survival probability. Indeed, the reproductive trade-off, incorporated by the penalty $\ve >0$, is always detrimental.  The same holds for the critical regime ($p=1/2$ and $\mu_X=1$), where both processes always go extinct: Here, the expected time to extinction is even finite for the two-type process $Z$.

However, in the sub-critical regime (3), while both processes do go extinct with probability 1, for small $w$ and $d$ (i.e.\ $Q_Z^2(1,0)$ and $Q_Z^2(0,0)$) the population with dormancy trait can be more likely to survive for extended periods of time, since by \eqref{eq:mot-ex-tails} $\prob[T_Z>n]>\prob[T_X>n]$ for $n\geq n_0$, i.e. $\prob[Z^1_n+Z^2_n>0]>\prob[X_n>0]$.  This is in line with basic intuition, since for small $w$ and $d$ individuals spend a long time in the dormant state delaying extinction.

This suggests that the `prolonged survival in the sub-critical regime' effect could lead to a fitness advantage in the presence of a random environment, fluctuating between a \emph{healthy} (super-critical) and a \emph{harsh} (sub-critical) scenario, even if the dormancy trait exhibits a reproductive trade-off in the healthy scenario, since dormancy could potentially compensate for this during harsh times by delaying extinction.  A central goal of this article is to identify and classify scenarios in which this is indeed the case. We thus now extend our models to incorporate such a fluctuating environment.

\subsection{Branching processes with dormancy in randomly fluctuating environment}
As indicated by Proposition~\ref{prop:motivating-exp} above, evolutionary fitness advantages resulting from a dormancy trait may be expected to manifest themselves in the presence of a random environment, where prolonged survival times may help to survive during harsh times. Here, we even expect strong fitness advantages in the sense that dormancy may turn an otherwise (overall) sub-critical process into a super-critical one, even in the presence of reproductive trade-offs.  We now introduce a simple model for a fluctuating environment, randomly oscillating between two states ``$1$'' and ``$2$'', where ``$1$'' corresponds to \emph{healthy} and ``$2$'' to \emph{harsh} conditions, which is identical to the environment given as an example in \cite[Sections~1.1.1 and 1.1.2.]{DMB11}. 
\begin{definition}[Binary random environment]
  \label{def:environment_process}
  Let $s_1, s_2 \in (0,1]$ and $s_1 \cdot s_2 < 1$. Define a discrete-time homogeneous Markov chain $(I_n)$ with values in $\{1,2\}$ via the transition matrix
  \begin{align*}
    P_I
    &\ldef\;
    \begin{pmatrix}
      1-s_1 & s_1\\
      s_2 & 1-s_2
    \end{pmatrix}
  \end{align*}
  where $s_1$ and $s_2$ denote the \emph{environmental switching probabilities}. Further, denote by $\pi_I=(s_2/(s_1+s_2),\ s_1/(s_1+s_2))$ the stationary distribution of $(I_n)$ and let $I_0\sim\pi_I$.
\end{definition}
\begin{remark}
  The initial condition for $I_0$ as well as the assertion $s_1s_2<1$ ensure that $(I_n)$ is stationary and ergodic.  Hence, the process $(I_n)$ is an example for a random environment process $(E_n)$ introduced at the beginning of this Section, taking only two values.  
\end{remark}
In the remainder of this section we will not give complete definitions of any further BGWPDRE's. We will only be interested in the mean matrices, while the exact offspring distributions will typically be irrelevant.  We will also restrict our attention to the above binary random environment. For environmental states $e \in \{1,2\}$, these matrices will be given by
\begin{align*}
  M(e)
  \;=\;
  \begin{pmatrix}
    \mAct^e & \mDor^e\\
    w^e & 1-w^e-d^e
  \end{pmatrix},
\end{align*}
where $\mAct^e$ and $\mDor^e$ represent the average amount of active and dormant offspring of active individuals respectively, while $w^e$ and $d^e$ denote resuscitation (`waking') and death probabilities of dormant individuals.  Our main results will concern two particular examples:

\begin{example}[Switching strategies]\label{ex:switching_strategies}
  The following idealized switching strategies represent important special cases that have been discussed in the literature, see e.g.\ \cite{MS08}, \cite{LJ11}, \cite{DMB11}. In particular, one distinguishes between \emph{responsive switching}, triggered by environmental conditions, and \emph{spontaneous} or \emph{stochastic switching}, which is assumed to happen in each individual with a certain probability, independently of the environmental states and the behaviour of the other individuals.
  \begin{itemize}
  \item[a)] \textbf{Responsive switching}:\\
    We consider the case where individuals behave ``optimally'' in the sense that they invest all their resources into the production of active individuals during the healthy environmental spells (choosing $\mDor^1=0$, $w^1=1-d^1$), while during harsh environmental conditions they invest everything into dormant offspring (choosing $\mAct^2=w^2=0$).  Hence, in this idealized case, the offspring mean matrices are of the form
    \begin{align*}
      M^{\mathrm{res}}(1)
      \;=\;
      \begin{pmatrix}
        m^1 & 0\\
        1-d^1 & 0
      \end{pmatrix}
      \qquad\text{and}\qquad
      M^{\mathrm{res}}(2)
      \;=\;
      \begin{pmatrix}
        0 & m^2\\
        0 & 1-d^2
      \end{pmatrix}
    \end{align*}
    for some $m^e > 0$ and $d^e < 1$.
  \item[b)] \textbf{Stochastic switching:}\\
    Here, the population assumes a reproductive strategy which is independent of its environmental state. This can be modeled by choosing $m_\bullet^1 = m_\bullet > 0$ and $m_\bullet^2 = \al m_\bullet$ for $\al \in [0,1)$ and $\bullet \in \{\text a, \text d\}$.  That way, $\mAct^1/\mDor^1 = \mAct^2/\mDor^2$, which means that active individuals split up their resources into the production of active and dormant offspring in the same way in both environments.  Then, the offspring mean matrices for $e \in \{1,2\}$ equate to
    \begin{align*}
      M^{\mathrm{sto}}(1)
      \;=\;
      \begin{pmatrix}
        \mAct & \mDor\\
        w^1 & 1-w^1-d^1
      \end{pmatrix}
      \qquad\text{and}\qquad
      M^{\mathrm{sto}}(2)
      \;=\;
      \begin{pmatrix}
        \al \mAct & \al \mDor\\
        w^2 & 1-w^2-d^2
      \end{pmatrix}
      % M(e)
      % \;=\;
      % \begin{pmatrix}
      %   \al^{e-1}\mAct & \al^{e-1}\mDor\\
      %   w_e & 1-w_e-d_e
      % \end{pmatrix}.
    \end{align*}
    Note that $\al < 1$ results in a reduced expected number of active offspring in the harsh environment. For $\al = 0$, no active individuals will be born at all during such conditions.
  \item[c)] \textbf{Prescient switching:}\\
    In this strategy, we consider individuals that invest all their resources into the production of dormant individuals during the healthy environmental spells (choosing $\mAct^1=w^1=0$), while during harsh environmental conditions they invest everything into active offspring (choosing $\mDor^2=0$, $w^2=1-d^2$). The corresponding offspring mean matrices are given by
    \begin{align*}
      M^{\mathrm{pre}}(1)
      \;=\;
      \begin{pmatrix}
        0 & m^1 \\ 0 & 1-d^1
      \end{pmatrix}
      \qquad\text{and}\qquad
      M^{\mathrm{pre}}(2)
      \;=\;
      \begin{pmatrix}
        m^2 & 0 \\ 1-d^2 & 0
      \end{pmatrix}
    \end{align*}
    for some $m^e > 0$ and $d^e < 1$.
  \end{itemize}
  In \emph{population genetic models} with seed bank, recently, similar types of switching have been distinguished (spontaneous vs.\ simultaneous switching) and these lead to topologically different limiting coalescent models describing the ancestry of a sample (see \cite{BEGCKW15}, \cite{BGCKW16}, \cite{BGKW18}). Here, in the presence of a random environment, we will see that the right choice of switching strategy can lead to qualitative fitness advantages, depending on the distribution of the environmental process.
\end{example}
Of course, less extreme variants, or even mixtures, of the above switching strategies should be interesting in practice. For example, as reported in \cite{VV15} and \cite{SD15}, phenotypic diversity in {\em Bacillus subtilis} with respect to the `exit from dormancy mechanisms' seems to combine stochastic switching of some individuals with responsive switching due to environmental cues of others on the population level.  However, the special form of the mean matrices in the above `pure'  strategies makes it possible to explicitly compute resp.\ obtain suitable bounds on the corresponding maximal Lyapunov exponents, which are crucial to assess and compare the fitness of the corresponding BGWDREs, as we will see below. Interestingly, these building blocks will later also allow the analysis of certain mixtures of strategies.
\begin{remark}[Comparison to multi-type branching process models considered in \cite{DMB11} and \cite{JW14}]
  Note that our above model is closely related to a very general multi-type branching process in random environment (MBPRE) modeling phenotypic diversity (with many, even a continuum of, possible types) considered in Dombry, Mazza and Bansaye \cite{DMB11}, and Jost and Wang \cite{JW14}, who themselves are inspired by the earlier work of Kussel and Leibler \cite{KL05}.  However, in their models, the authors follow a \emph{ two step procedure}, where in a first step each particle gives birth to a random amount of offspring (depending on its type and the state of the environment), and then in a second step, \emph{independently} of that amount, the offspring particles are fitted (individually) with their new phenotypes.  As the authors point out, this clearly disentangles the birth and ``migration'' (between phenotypes) phases.  The BGWPDRE-model that we propose here is tailored to dormancy and does not disentangle these steps.  This has consequences for the possible switching strategies. In fact it turns out that some of our reproductive strategies presented above are not covered by the framework of \cite{DMB11} and \cite{JW14}.  For example, they do not cover the case that active offspring in the healthy environment can either split into two active offspring (cell division), or switch to a dormant state (e.g.\ by sporulation), as in the example of Section~\ref{sec:mot-exp}, since the phenotype distribution in this case would have to be allowed to depend on the number of offspring of the parent, see also Remark~\ref{rem:discussion:DMB11}.  
\end{remark}
\begin{remark}[Comparison to switching strategies employed in \cite{MS08}]
  Malik and Smith in \cite{MS08} consider a related, but less general switching model. Again, there are two possible environmental states, however, the bad environment here always completely prevents the reproduction of active individuals. Exact analytic expressions for the Lyapunov exponents are obtained only for the case where the environment is deterministic. In the random environment case, still some bounds are provided. It turns out that we can adapt the corresponding methods  to obtain bounds for Lyapunov exponents of our models (cf.\ Remark~\ref{rem:further-bounds}).
\end{remark}

\subsection{Asymptotic growth of BGWPDREs and Lyapunov exponents}\label{sec:fitness_BGWDRE}
Of particular interest is the asymptotic behaviour of the process $Z$.  It is well known, cf.\ \cite{Ka74}, that
\begin{align*}
  \mean\!\big[Z_n \,\big|\, Z_0,\, \Pi = (Q_Z(e_1), Q_Z(e_2), \ldots)\big]
  \;=\;
  Z_0 \cdot M_1 \cdot \ldots \cdot M_n.
\end{align*}
The study of such products of random matrices has a long and venerable history dating back to first results by Furstenberg and Kesten \cite{FK60}.  For stationary and ergodic sequences $(M_1, M_2, \ldots)$ of non-negative matrices satisfying $\mean\!\big[\log^+ \|M_1\|\big] < \infty$, where $\log^+ x \ldef \max\{\log x, 0\}$, Kingman \cite{Ki73}, see also \cite{Os68}, proved that, by means of the subadditive ergodic theorem,
\begin{align}
  \vp
  \;\ldef\;
  \lim_{n \to \infty} \frac{1}{n} \log \|M_1 \cdot \ldots \cdot M_n\|
  \;\in\;
  [-\infty, \infty)
\end{align}
exists $\prob$-a.s.\ and also satisfies
\begin{align*}
  \vp
  \;=\;
  \lim_{n \to \infty} \frac{1}{n}
  \mean\!\Big[\log \|M_1 \cdot \ldots \cdot M_n\| \Big].
\end{align*}
In particular, $\vp$ is independent of the chosen matrix norm.  The limit, $\vp$, is called \emph{maximal Lyapunov exponent}.  
\begin{remark}[Exact computation of Lyapunov exponents]\label{rem:lyap-exp}\hspace{-1.25ex}
  There are only a few cases where the maximal Lyapunov exponent can be computed explicitly.  For instance, if $(M_1, M_2, \ldots)$ is a stationary and ergodic process of positive $1 \times 1$ matrices, i.e.\ $M_n = \vr(M_n)$, with $\mean[\log^+ \vr(M_1)] < \infty$, then an application of Birkhoff's ergodic theorem yields, $\prob$-a.s.,
  \begin{align}
    \vp
    \;=\;
    \lim_{n \to \infty} \frac{1}{n} \sum_{k=1}^n \log M_k
    \;=\;
    \mean\!\big[\log \vr(M_1) \big].
    \label{eq:1t-fitness}
  \end{align}
  A further simple case is given by a sequence of stationary and ergodic matrices with $\mean[\log^+ \|M_1\|] < \infty$ such that the matrices $M_i$ are either mutually diagonizable, i.e.\ $[M_i, M_j] = 0$ for all $i \ne j$, or of upper (lower) triangular form.  Then, $\prob$-a.s.,
  \begin{align*}
    \vp
    \;=\;
    \lim_{n \to \infty} \frac{1}{n} \log \|M_1 \cdot \ldots \cdot M_n\|
    \;=\;
    \mean\!\big[\log\vr(M_1)\big].
  \end{align*}
  Recall that we denote by $\vr(M)$ the spectral radius of the matrix $M$.  Further cases in which the maximal Lyapunov exponent can be computed explicitly are discussed in \cite{Ke87}.  For the general case, where the computation of $\vp$ is difficult resp.\ infeasible, there are various strategies for giving bounds known in the literature, see also \cite{CPV93}. We will discuss and employ possible methods in Section~\ref{sec:rank-2-case}.
\end{remark}
\begin{remark}[Approximation of Lyapunov exponents]\label{rem:simple:bounds}\hspace{-1.25ex}
  Under certain further assumptions on the stationary and ergodic sequence, $(M_1, M_2, \ldots)$, of non-negative matrices with $\mean[\log^+ \|M_1\|] < \infty$, Key \cite{Ke90} proved that, $\prob$-a.s.\ and in mean,
  \begin{align*}
    \vp
    \;=\;
    \lim_{n \to \infty} \frac{1}{n} \log f(M_1 \cdot \ldots \cdot M_n)
  \end{align*}
  for any one-homogeneous, non-negative, super-multiplicative function, $f$, such that $\mean[\log^- f(M_1)] > -\infty$. By defining
  \begin{align*}
    \underline{\vp}_k
    \;\ldef\;
    \frac{1}{k} \mean\!\big[\log f(M_1 \cdot \ldots \cdot M_k)\big]
    \qquad \text{and} \qquad
    \overline{\vp}_k
    \;\ldef\;
    \frac{1}{k} \mean\!\big[\log \|M_1 \cdot \ldots \cdot M_k\|\big],
  \end{align*}
  then it follows from the sub-multiplicativity of $\| \cdot \|$, the super-multiplicativity of $f$, and the stationarity of the sequence $(M_1, M_2, \ldots)$ that $\underline{\vp}_k$ increases monotonically to $\vp$, whereas $\overline{\vp}_k$ decreases monotonically to $\vp$.  Although this provides an easy way to derive upper and lower bounds on the maximal Lyapunov exponent, the computational effort increases exponentially in $k$.  For i.i.d. sequences $(M_1, M_2, \ldots)$ Pollicot \cite{Po10} and Jurga and Masion \cite{JM19} established efficient approximation schemes with super-exponential convergence rates, see also \cite{PJ13} for further bounds.
\end{remark}
Notice that for a $p$-type BGWPRE, $Z$, as defined above, the sequence of mean matrices, $(M_1, M_2, \ldots)$, form a stationary and ergodic process.  Thus, provided that $\mean\!\big[\log^+ \|M(E_1)\|\big] < \infty$, the corresponding maximal Lyapunov exponent, $\vp_Z$, exists describing the asymptotic rate of growth/decay of the expected value of $Z$.

The almost sure behaviour of the process $Z$ has also been studied intensively.  For instance, if $Z$ is a $p$-type BGWPRE such that $M_n \in (0, \infty)^{p \times p}$ for all $n \in \bbN$ and $\mean\!\big[\log^+ \|M(E_1)\| \big] < \infty$, then it follows from \cite[Theorem~9.10]{Ta81} that, for almost all realizations of the environment, $\vp_Z < 0$ implies that $Z$ becomes extinct almost surely, whereas for $\vp_Z > 0$ there exists a positive probability that $Z$ never becomes extinct.  Moreover, conditioned on survival, we have that
\begin{align*}
  \lim_{n \to \infty} \frac{1}{n} \log \|Z_n\|_1 \;=\; \vp_Z
\end{align*}
almost surely.  In particular, the almost sure growth of the stochastic switching model, cf.\ Example~\ref{ex:switching_strategies}-b), conditioned on non-extinction, is given by $\vp_Z$.

Actually, Tanny established in \cite[Theorem~9.6 and Theorem~9.10]{Ta81} a classification theorem for more general multi-type BGWPRE with non-negative mean matrices satisfying certain regularity conditions.  Notice that these conditions are not satisfied by our responsive and prescient switching model, cf.\ Example~\ref{ex:switching_strategies}-a).  However, due to the particular structure that allows a reduction of this BGWPDRE to a $1$-type BGWPRE, cf.\ \cite[Proposition~7]{DMB11}, an analogous classification theorem can then be deduced from \cite[Theorem~9.6]{Ta81}, see also \cite{AK71a, AK71b}.
\begin{remark}[Lyapunov exponent, fitness and survival-probability]\label{rem:measures-of-fitness}
  The previously mentioned features of the maximal Lyapunov exponent, describing various growth properties of population models, justifies the use of $\vp$ as a measure of \emph{fitness} of population models, as is common in the literature. However, there is no direct monotone relationship between $\vp$ and the \emph{survival probability} of the underlying population in the super-critical case, as the following example confirms: Taking the setting from Section~\ref{sec:mot-exp}, choosing $X$ with parameter $p=4/7$, one can compute that $\si_X=2-\frac{1}{p}=0.25$ and $\vp_X=2p\approx1.143$. Then, for $Z$ letting $p=4/5$, $\ve=0$, $b=2/5$, $w=1/2$ and $d=1/25$ it holds that
  \begin{align*}
    \si_Z
    \;=\;
    2 - \frac{1}{bp} + \frac{1-b}{b} \cdot \frac{w}{w+d}
    \;\approx\;
    0.264
    \;>\;
    \si_X.
  \end{align*}
  However (cf.\ \eqref{eq:rho-lower-bound} below), $\vp_Z \approx 1.050 < \vp_X$. Hence, the comparison of Lyapunov exponents of distinct population models does not necessarily give a complete picture of the advantages of one model over the other.

  This phenomenon has been studied in more detail by Jost and Wang \cite{JW14}, where the authors illustrate that different optimization criteria (i.e.\ largest growth rate vs.\ smallest extinction probability) can lead to distinct optimal strategies.
\end{remark}
We now move on to some of the main results of this paper.  Indeed, the Lyapunov exponent for a BGWPDRE with responsive and prescient switching strategy can be computed explicitly:
\begin{theorem}[Lyapunov exponent of the responsive switcher]
  \label{thm:responsive-fitness}
  Let $(Z_n)$ be a BGWPDRE with environment process $(I_n)$ from Definition~\ref{def:environment_process}, following the \emph{responsive} switching regime in Example~\ref{ex:switching_strategies}, a). Then, $\prob$-a.s.,
  \begin{align*}
    \vp_Z
    \;=\;
    \frac{
      s_2 \log m^1 + s_1\log(1-d^2)
      + s_1 s_2\log\Big(\frac{m^2(1-d^1)}{m^1(1-d^2)}\Big)}
    {s_1+s_2}.
  \end{align*}
\end{theorem}
\begin{theorem}[Lyapunov exponent of the prescient switcher]
  \label{thm:prescient-fitness}
  Let $(Z_n)$ be a BGWPDRE with environment process $(I_n)$ from Definition~\ref{def:environment_process}, following the \emph{prescient} switching regime in Example~\ref{ex:switching_strategies}, b). Then, $\prob$-a.s.,
  \begin{align*}
    \vp_Z
    \;=\;
    \frac{
      s_2 \log (1-d^1) + s_1\log m^2
      + s_1 s_2\log\Big(\frac{m^1(1-d^2)}{m^2(1-d^1)}\Big)}
    {s_1+s_2}.
  \end{align*}
\end{theorem}
A proof will be provided in the next section. For the stochastic switcher we obtain the following analytic result under the additional assumption that the determinant of the mean matrices vanishes:
\begin{theorem}[Lyapunov exponent of the stochastic switcher]
  \label{thm:stochastic-fitness}
  Let $(Z_n)$ be a BGWPDRE with environment process $(I_n)$ from Definition~\ref{def:environment_process}, following the \emph{stochastic} switching regime in Example~\ref{ex:switching_strategies}, b) with 
  $\det M(1) = \det M(2) = 0.$
   Then, $\prob$-a.s.,
  \begin{align}
    \vp_Z
    \;=\;
    \frac{
      s_2 \log\big(\mAct + w^1 \frac{\mDor}{\mAct}\big)
      + s_1 \log\big(\al \mAct + w^2 \frac{\mDor}{\mAct}\big)
    }{s_1+s_2}.
  \end{align}
\end{theorem}
These results are closely related to results in \cite{DMB11}, considering that mean matrices of determinant zero correspond to the \emph{non-hereditary-with-sensing} case therein (cf.\ Section~\ref{sec:rank-one-case} for the proof and further remarks). In the \emph{hereditary} case, i.e.\ the case of non-zero determinants, neither \cite{DMB11} nor the article at hand obtain an explicit result for $\varphi_Z$. However, various bounds will be discussed in Section~\ref{sec:rank-2-case}. We provide one of them here in a special case, for illustration:

\begin{theorem}\label{thm:dirac-lower-bound}
  Let $(Z_n)$ be a BGWPDRE with environment process $(I_n)$ from Definition~\ref{def:environment_process}, following the \emph{stochastic} switching regime in Example~\ref{ex:switching_strategies} with $w^1 = w^2$, $d^1 = d^2$ and $\al \in (0,1)$.  Then, $\prob$-a.s.,
  \begin{align*}
    \vp_Z
    \;\geq\;
    \mean\!\big[ \log\big(\tr M(I_0)-\max\{\det M(1)/\mAct,0\}\big) \big].
  \end{align*}
\end{theorem}
Note that since $w$ and $d$ do not depend on $e$, we get $\det M(2) = \al \det M(1)$. Notably, when $\det M(1)=0$, this lower bound equates to the result from Theorem~\ref{thm:stochastic-fitness}. 

Further bounds will be provided in the next section, where we also try to shed light on the structures of switching mechanisms that allow for the computation of analytical results and bounds. Indeed, we will distinguish the so-called `rank-1'-case (which is closely related to the results in \cite{DMB11}), allowing explicit computations, and the `rank-2'-case, where often only bounds can be provided. Here, we refer to the rank of mean matrices of the reproduction resp.\ switching mechanisms. Obviously, the mean matrices of the responsive and prescient switcher in Example~\ref{ex:switching_strategies} are degenerate and of rank 1, as are the mean matrices of the stochastic switcher in Theorem~\ref{thm:stochastic-fitness}, due to the vanishing determinant, while the stochastic switcher of Theorem~\ref{thm:dirac-lower-bound} has mean matrices of rank 2. Yet, this switching mechanism also has particular properties that will be exploited in the next section.

Before we carry out these considerations and prove the above results, we first investigate the selective advantages of the switching strategies in different environments.

\subsection{Fair comparison of BGWPDREs with different switching strategies}
To decide which switching strategy of two different BGWPDREs is superior in an environment given by $(I_n)$, one needs to impose a condition that ensures that both processes ``may use an equal amount of available resources''.  One way to do this would be to require that both processes can produce in expectation the same amount of offspring in each generation, be it active or dormant offspring, and to assume that the death probabilities of both processes are the same in both the active and dormant states each.  The processes thus can adapt to the environment only by means of their specific switching strategies while using the same amount of resources.  This motivates our notion of fitness advantages under ``fair comparison''. We formulate this concept in the general framework of $p$-type branching processes in random environemnts $(E_n)$.
\begin{definition}[Fitness advantage under fair comparison]
  \label{def:fair-comparison}
  For $p \geq 1$ let $Z \equiv (Z_n)$ and $\bar{Z} \equiv (\bar Z_n)$ two $p$-type BGWPRE with respect to the same environmental process $(E_n)$ such that $\prob$-a.s.\ for all $1 \leq t \leq p$ and $n \geq 1$ it holds for their mean matrices that
  \begin{align}\label{eq:fair-comp-condition}
    \sum_{i=1}^p m_n^{t,i}
    \;=\;
    \sum_{i=1}^p \overline{m}_n^{t,i}.
  \end{align}
  Then, if $\vp_Z > \vp_{\bar Z}$, we say that $Z$ is \emph{fitter} than $\bar{Z}$ at \emph{fair comparison}.  If additionally $\vp_Z > 0 \geq \vp_{\bar Z}$, we say that  $Z$ has a \emph{strong} (or qualitative) fitness advantage over $\bar Z$ under fair comparison.
\end{definition}
\begin{remark}\label{rem:def-fair-comp}
  \begin{enumerate}
  \item The concept of Definition~\ref{def:fair-comparison} is in the same spirit as the comparison of strategies in \cite{DMB11}, since equation~\eqref{eq:fair-comp-condition} assures that for each $t$, type-$t$-individuals in both populations produce in expectation the same amount of offspring, only varying the distribution of types among offspring.
  \item For BGWPDREs in environment $(I_n)$, Equation~\eqref{eq:fair-comp-condition} is equivalent to
    \begin{align*}
      &(i)\quad
        \mAct^e + \mDor^e
        \;=\;
        \overline{m}_{\mathrm{a}}^e + \overline{m}_{\mathrm{d}}^e
      &\text{and}&
      &(ii)\quad d^e \;=\; \bar{d}^e
    \end{align*}
    for each $e \in \{1,2\}$.
  \item To allow a comparison of a BGWPDRE to a 1-type process (i.e.\ without dormancy), let $(X_n)$ be a 1-type BGWPRE with environment $(I_n)$ with conditional offspring means $m_1$ and $m_2$ (referring to healthy and harsh environmental states respectively).  This process can be understood as a 2-type BGWPRE process in the sense of Definition~\ref{def:bgwdre}, starting in $(1,0)$, and having mean matrices
    \begin{align*}
      M(e)
      \;=\;
      \begin{pmatrix}
        m^e & 0\\
        1-d^e & 0
      \end{pmatrix}
    \end{align*}
    for $e \in \{1,2\}$ and some arbitrary $d^e \in (0,1)$.  Although these matrices are reducible, this makes a fair comparison feasible.
  \item Note that the notion of fair comparison alone does not yet imply any kind of reproductive trade-off. However, every Lyapunov exponent or bound of such we will compute in the rest of this paper is continuous in the model parameters. This continuity and the strictness of the inequality in the definition of fitness advantages very well include the possibility of advantages under `disadvantageous' comparison given (sufficiently small but non-trivial) trade-offs.
  \end{enumerate}
\end{remark}
One of the main goals of this article is to identify situations, under fair comparison, in which one switching strategy can be super-critical, whereas the other switching strategy or the process without dormancy is sub-critical. Note that this is impossible in the absence of a random environment, as pointed out in the discussion after Proposition~\ref{prop:motivating-exp}. This is now obtained in the context of fair comparison and making use of Remark~\ref{rem:def-fair-comp}:
\begin{theorem}\label{thm:advantages-of-seedbanks}
  Denote by $(I_n)$ an environment process as in Definition~\ref{def:environment_process}, by $X$ a 1-type branching process as in Remark~\ref{rem:def-fair-comp} and by $Z^{\mathrm{res}}, Z^{\mathrm{pre}}, Z^{\mathrm{sto}}$ BGWPDREs following either the responsive, prescient or stochastic switching strategy as in Example~\ref{ex:switching_strategies}. Then, for either of the four processes there are non-trivial parameter regimes and environmental distributions, under which this process has a strong fitness advantage over the other three in the sense of Defintion~\ref{def:fair-comparison}.
\end{theorem}
We prove this Theorem by means of examples of dominant strategies combining the results of Theorem~\ref{thm:responsive-fitness}, Theorem~\ref{thm:prescient-fitness}, Theorem~\ref{thm:stochastic-fitness} and \eqref{eq:1t-fitness} from Remark~\ref{rem:lyap-exp} after fitting the parameters to the regime of fair comparison.
\begin{figure}
  \includegraphics{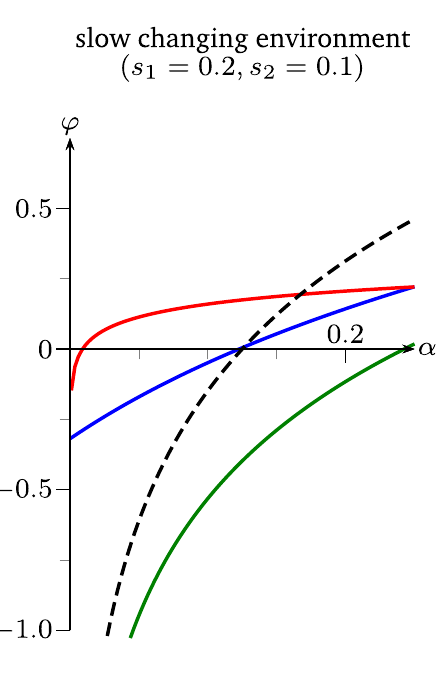}
  \hspace{4mm}
  %\includegraphics{04d_fig}
  %\hspace{4mm}
  \includegraphics{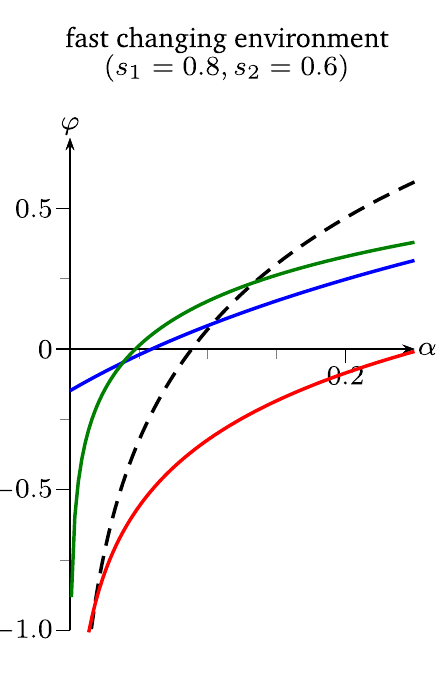}
  \caption{\label{fig:fit-adv_fair-comp} Parameter regimes of Example~\ref{exps:strong-advantages} (1) and (3) respectively, Lyapunov exponents taken as functions of $\al$. black: $\vp_X$, red: $\vp_{\mathrm{res}}$, blue: $\vp_{\mathrm{sto}}$, green: $\vp_{\mathrm{pre}}$.}
\end{figure}
\begin{example}[Strong fitness advantages of seed bank switching strategies]
  \label{exps:strong-advantages}\hspace{-.75ex}
  Let $X$ be a $1$-type BGWPRE as in Remark~\ref{rem:def-fair-comp} above with $m(1) = 4$ and $m(2) = 4\al$, where $\al < 1/4$ such that $X$ is sub-critical in the second environment.  Further, let $Z^{\mathrm{res}}$, $Z^{\mathrm{pre}}$ and $Z^{\mathrm{sto}}$ be three BGWPDREs with mean matrices
  \begin{align*}
    M^{\mathrm{res}}(1)
    &\;=\;
    \begin{pmatrix}
      4   & 0\\
      4/5 & 0
    \end{pmatrix},
    &
    M^{\mathrm{res}}(2)
    &\;=\;
    \begin{pmatrix}
      0 & 4\al\\
      0 & 4/5
    \end{pmatrix},          
    \\[1ex]
    M^{\mathrm{pre}}(1)
    &\;=\;
    \begin{pmatrix}
      0   & 4\\
      0 & 4/5
    \end{pmatrix},
    &
    M^{\mathrm{pre}}(2)
    &\;=\;
    \begin{pmatrix}
      4\al & 0\\
      4/5 & 0
    \end{pmatrix},          
    \\[1ex]
    M^{\mathrm{sto}}(1)
    &\;=\;
    \begin{pmatrix}
      2 & 2\\
      2/5 & 2/5
    \end{pmatrix},
    &
    M^{\mathrm{sto}}(2)
    &\;=\;
    \begin{pmatrix}
      2\al & 2\al\\
      2/5 & 2/5
    \end{pmatrix}.
  \end{align*}
  Noting that both the responsive and prescient switching matrices correspond to $d^1 = d^2 = 1/5$ and the stochastic switching matrices additionally to $w^1 = w^2 = 2/5$.  In particular, these processes satisfy \eqref{eq:fair-comp-condition}, the condition of fair comparison. Also note that $\det M^{\mathrm{sto}}(1) = \det M^{\mathrm{sto}}(2) = 0$, such that we obtain an exact result from Theorem~\ref{thm:stochastic-fitness}.
  
  Now, only the environment-related parameters $\al < 1/4$ and $s_1, s_2$ are left to play with, describing the severity of harsh environments and the lengths of the environmental phases. The following cases prove Theorem~\ref{thm:advantages-of-seedbanks}:
  \begin{enumerate}
  \item For $\al = 1/20$, $s_1 = 2/10$ and $s_2 = 1/10$ we obtain
    \begin{align*}
      &\vp_X \approx -0.61 < 0,&      
      &\vp_{\mathrm{sto}} \approx -0.17 < 0,&
      &\vp_{\mathrm{pre}} \approx -0.95 < 0&                                   
      &\text{but}&
      &\vp_{\mathrm{res}} \approx 0.11 > 0.
    \end{align*}
  \item For $\al = 1/20$, $s_1 = 1/2$ and $s_2 = 1/2$ we get
    \begin{align*}
      &\vp_X \approx -0.11 < 0,&
      &\vp_{\mathrm{res}} \approx -0.17 < 0,&
      &\vp_{\mathrm{pre}} \approx -0.17 < 0&
      &\text{but}&
      &\vp_{\mathrm{sto}} \approx 0.09 > 0.
    \end{align*}
  \item Letting $\al = 1/20$, $s_1 = 8/10$ and $s_2 = 6/10$ implies
    \begin{align*}
      &\vp_X \approx -0.33 < 0,&
      &\vp_{\mathrm{res}} \approx -0.56 < 0,&
      &\vp_{\mathrm{sto}} \approx -0.02 < 0&
      &\text{but}&
      &\vp_{\mathrm{pre}} \approx 0.01 > 0.
    \end{align*}
  \item Finally, choosing $\al = 1/5$, $s_1 = 8/10$ and $s_2 = 3/20$ yields
    \begin{align*}
      &\vp_{\mathrm{res}} \approx -0.17 < 0,&
      &\vp_{\mathrm{sto}} \approx -0.05 < 0,&
      &\vp_{\mathrm{pre}} \approx -0.02 < 0&
      &\text{but}&
      &\vp_X \approx 0.03 > 0.
    \end{align*}
  \end{enumerate}
\end{example}
\begin{figure}
  \includegraphics{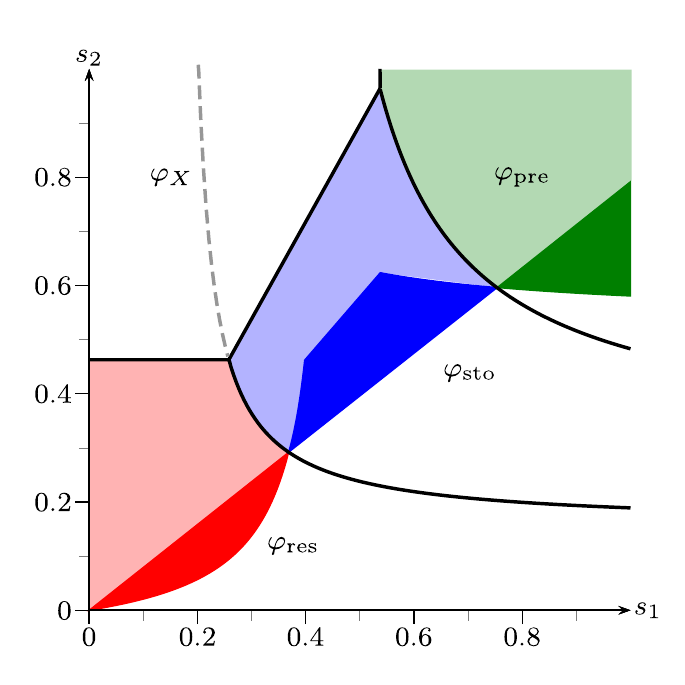}
  \hspace{-0.2cm}
  \includegraphics{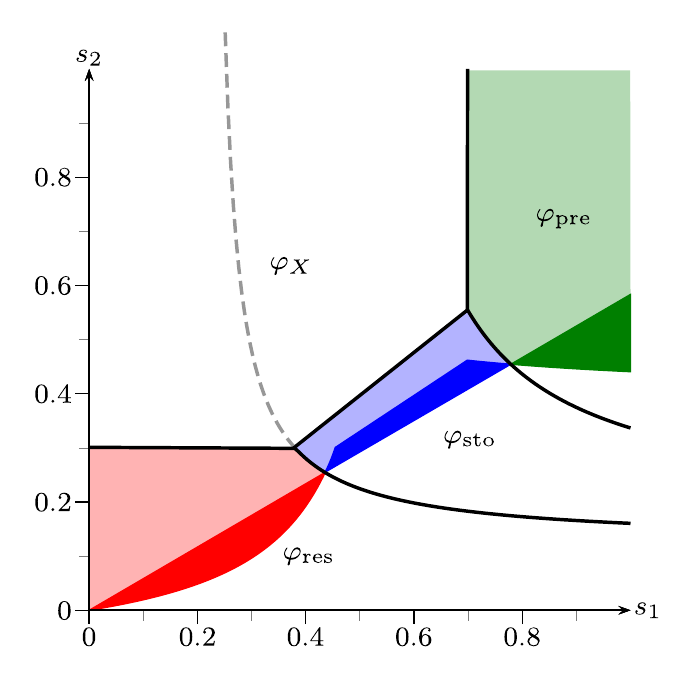}
  \caption{\label{fig:phase_diagram} Phase diagram of the maximal Lyapunov exponents $\vp_X$, $\vp_{\mathrm{res}}$, $\vp_{\mathrm{sto}}$, $\vp_{\mathrm{pre}}$ of Example~\ref{exps:strong-advantages} with $\al = 1/20$ (left) and $\al = 1/10$ (right).  Strong advantage of $\vp_{\mathrm{res}}$ (red), $\vp_{\mathrm{sto}}$ (blue), and $\vp_{\mathrm{pre}}$ (green). Advantegous and supercritical: light red, light blue, light green resp.}
\end{figure}
\begin{remark}[Interpretation of advantageous strategies]
  \label{rem:interpr-advs}\hspace{-1.25ex}
  Figure~\ref{fig:fit-adv_fair-comp} provides more insight into the behaviour of the four strategies than Example~\ref{exps:strong-advantages}, by taking the parameter regimes (1) and (3) thereof and plotting the respective Lyapunov exponents as functions of $\al \in (0, 1/4)$.  Furthermore, Figure~\ref{fig:phase_diagram} shows the fitness advantage landscapes of the four models as a function of $(s_1,s_2)$, where strong advantages are colorized.

  The responsive switcher, when $m_2 \ll 1-d_2$, suffers most upon entering or exiting the harsh environment.  Hence, in a scenario where environments rarely change (cf.\ Figure~\ref{fig:fit-adv_fair-comp} (left) and Figure~\ref{fig:phase_diagram}), responsive switching does well compared to the other strategies.

  The stochastic switcher however, performs a \emph{bet hedging} strategy, i.e.\ investing in dormant offspring even in good times to have better chances in worse times. This can often be very costly, but really pays off when environments change with a moderate frequently (cf.\ Figure~\ref{fig:phase_diagram}), especially when bad environments get very harsh, i.e.\ when $\al$ is small.

  The prescient switcher is an extreme bet hedging strategy since it invests under good environmental conditions all its resources in producing dormant offspring, whereas in bad environments it only produces active offspring (cf.\ Figure~\ref{fig:fit-adv_fair-comp} (right) and Figure~\ref{fig:phase_diagram}). %Nevertheless, in a scenario where environmental conditions change with a high frequency, this prescient switcher outperforms the other strategies.
Arguably, this can be understood as a \emph{counter-intuitive responsive switching strategy}. Nevertheless, in the extreme case $s_1=s_2=1$, where deterministically the environment changes at every generation, it is intuitive that the prescient switcher is optimal. Figure~\ref{fig:phase_diagram} even shows a non-trivial parameter region, in which this strategy is dominant.
  
  Note that the $1$-type process without dormancy trait will always dominate the switching strategies when $\al$ becomes sufficiently big -- i.e.\ when the process gets less and less sub-critical in bad environments -- as illustrated in Figure~\ref{fig:phase_diagram} (right). In fact, in that particular parameter setting, when $\al \geq 1/5$, only the region of $\vp_X$ will appear in the phase diagram, meaning that for any switching parameters $\vp_X\geq\max\{\vp_{\mathrm{res}}, \vp_{\mathrm{sto}}, \vp_{\mathrm{pre}}\}$.  This corresponds to Proposition~\ref{prop:motivating-exp} from the beginning of this paper, where we saw that seed bank strategies are at a disadvantage in super-critical environments.

  Lastly, note that, for general values of $\al$, the case of iid environments -- which corresponds to the line on which $s_1+s_2=1$ -- would not at all capture the strong advantage of responsive and prescient switching in the settings of Figure~\ref{fig:phase_diagram}. Hence, for providing a complete understanding of the fitness landscapes, the iid case is insufficient.
\end{remark}
\begin{figure}[t]
  \includegraphics{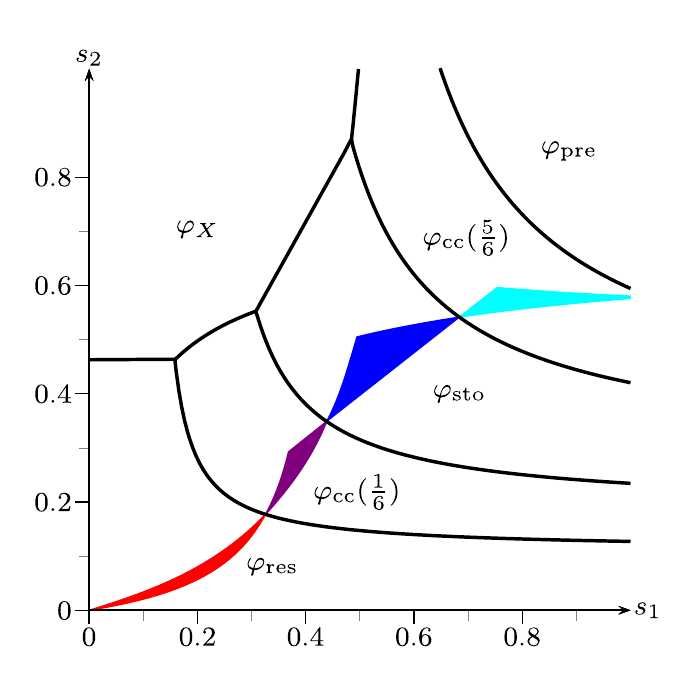}
  \hspace{-0.2cm}
  \includegraphics{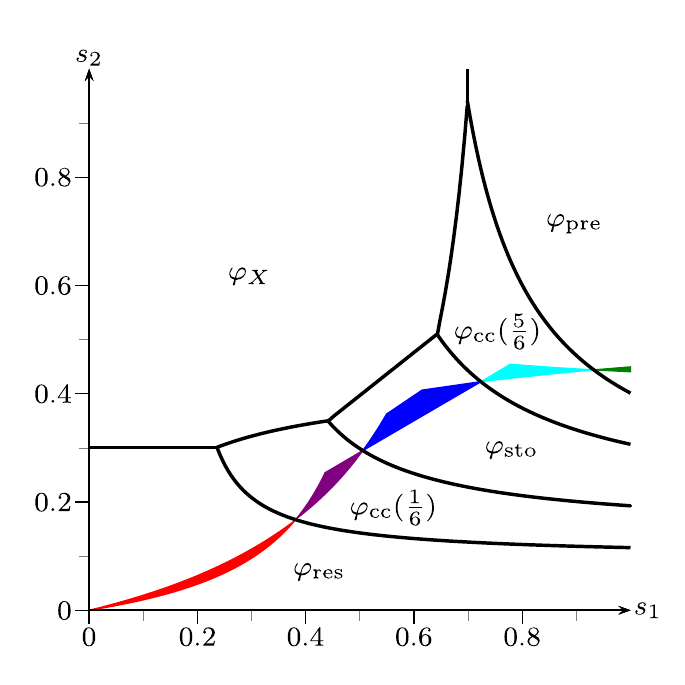}
  \caption{\label{fig:phase_diagram_cc} Phase diagram of the maximal Lyapunov exponents $\vp_X$, $\vp_{\mathrm{res}}$, $\vp_{\mathrm{sto}}$, $\vp_{\mathrm{pre}}$, $\vp_{\mathrm{cc}}(1/6)$ and $\vp_{\mathrm{cc}}(5/6)$ of Example~\ref{exps:strong-advantages} with $\al = 1/20$ (left) and $\al = 1/10$ (right), and Remark~\ref{rem:convex-combination} with $q=1/6$ and $q=5/6$.  Strong advantage of $\vp_{\mathrm{res}}$ (red), $\vp_{\mathrm{sto}}$ (blue), $\vp_{\mathrm{pre}}$ (green), $\vp_{\mathrm{cc}}(1/6)$ (purple), and $\vp_{\mathrm{cc}}(5/6)$ (cyan).}
\end{figure}
\begin{remark}[Combining basic strategies]\label{rem:convex-combination}
  The presence of phenotypic diversity regarding different switching strategies within the same Bacillus population (at least wrt.\ the exit strategy from dormancy, see \cite{VV15} and \cite{SD15}) suggests also to investigate in mixtures of switching strategies.  To model this, we let each indivi\-dual choose at birth whether it behaves according to the prescient (with probability $q(e) \in [0,1]$) or the responsive (with probability $1-q(e)$) switching mechanism.  The resulting mean matrices are given by
  \begin{align*}
    M_q^{\mathrm{cc}}(e)
    \;\ldef\;
    q(e) M^{\mathrm{pre}}(e) + (1-q(e)) M^{\mathrm{res}}(e),
  \end{align*}
  still maintaining fair comparison.  Furthermore, these matrices also have determinant 0. (In fact, linear combinations of rank-1-matrices under fair comparison always retain rank 1).  In particular, under fair comparison, the stochastic switcher with $\det M^{\mathrm{sto}}(1) = \det M^{\mathrm{sto}}(2) = 0$ can be represented as the convex combination of the responsive and prescient switcher with $q(1) = \mDor/(\mAct+\mDor)$ and $q(2)=\mAct/(\mAct+\mDor)$.  Hence, we can again compute their fitness explicitly, and this leads to very interesting behaviour.
  
  Inserting the matrices of Example~\ref{exps:strong-advantages}, we obtain for $q(1)=q(2)=q$
  \begin{align*}
    M_q^{\mathrm{cc}}(1)
    \;=\;
    \begin{pmatrix}
      4 (1-q) & 4 q \\
      4(1-q)/5 & 4q/5
    \end{pmatrix}
    \qquad\text{and}\qquad
    M_q^{\mathrm{cc}}(2)
    \;=\;
    \begin{pmatrix}
      4\al q & 4\al (1-q) \\
      4q/5 & 4(1-q)/5
    \end{pmatrix}.
  \end{align*}

  Figure~\ref{fig:phase_diagram_cc} illustrates -- in comparison to Figure~\ref{fig:phase_diagram} -- which influence the convex combination of the basic strategies can have. Very intuitively, the region where $\vp_{\mathrm{cc}}(1/6)$ and $\vp_{\mathrm{cc}}(5/6)$ has an advantage lies between the regions of the responsive/stochastic and stochastic/prescient strategies.  Remarkably, e.g.\ around $(s_1,s_2)=(0.4,0.3)$ for $\al=1/20$ there is even a region where the convex combination of the responsive and prescient strategy with $q=1/6$ yields a \emph{strong} advantage possibly preventing extinction which, however, is certain for the responsive, stochastic and prescient strategy.

  This can be motivated as follows: For $s_1,s_2$ both either small or large, one of the pure strategies (responsive or prescient) seems to be optimal. However, for moderate $s_1,s_2$ environmental variation is high and both fast switching and slow switching environmental phases might occur.  Then, a combination of both strategies ensures that the worst case for neither strategy can affect the whole population.  If one considers the strategy of stochastic switching as a bet-hedging strategy, then using phenotypic diversity to employ a mixture of extreme strategies might be seen as a `second-level' bet-hedging strategy, now with respect to switching behaviour.
  %\footnote{\textcolor{blue}{MS: I am not quite sure if the last sentence make sense. I think that there is not real reason anymore to talk about a 'second-level' bet-hedging strategy.}}
\end{remark}
\begin{figure}
  \includegraphics{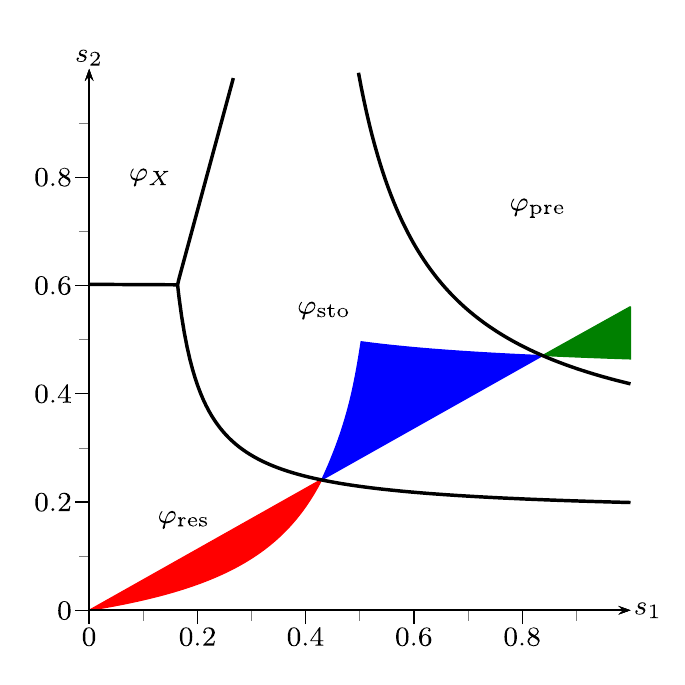}
  \hspace{-0.2cm}
  \includegraphics{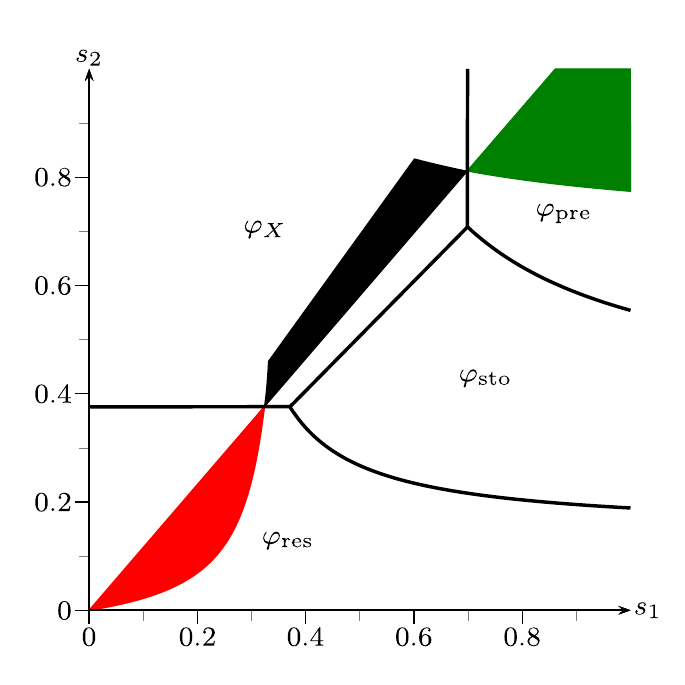}
  \caption{\label{fig:phase_diagram:fair} Phase diagram of the maximal Lyapunov exponents $\vp_X$, $\vp_{\mathrm{res}}$, $\vp_{\mathrm{sto}}$, $\vp_{\mathrm{pre}}$ of Remark~\ref{rem:weighted-fair-comp} with $\al = 1/20$, $\ga = 1/2$ (left) and $\ga = 2$ (right). Strong advantage of $\vp_{\mathrm{res}}$ (red), $\vp_{\mathrm{sto}}$ (blue), $\vp_{\mathrm{pre}}$ (green), and $\vp_{X}$ (black).}
\end{figure}
\begin{remark}[Refinement of fair comparison]\label{rem:weighted-fair-comp}
  The notion of fair comparison implicitly assumes that the production of or conversion into dormant forms is equally costly as the production of active offspring. In many scenarios this will not be realistic. In fact, the production of inactive individuals can be both very efficient  (e.g.\ in seed plants) as well as rather costly (e.g. sporulation of \emph{Bacillus subtilis}, see \cite{PH04}).  Exchanging \eqref{eq:fair-comp-condition} in Definition~\ref{def:fair-comparison} by
  \begin{align*}
    \sum_{i=1}^p m_n^{t,i} \ga^i \;=\; \sum_{i=1}^p \overline{m}_n^{t,i} \ga^i
  \end{align*}
  for some $\ga \in (0,\infty)^p$ leads us to the notion of ``$\ga$-weighted fair comparison''.  For BGWDPREs in the environment $(I_n)$ we are mainly interested in $\ga$-weighted fair comparison with $(\ga^1,\ga^2)=(1,\ga)$ for some $\ga>0$.  Hence, the condition above reads
  \begin{align}
    &(i)\quad
    \mAct^e + \ga \mDor^e
    \;=\;
    \overline{m}_{\textnormal a}^e + \ga \overline{m}_{\textnormal d}^e
    &\text{and}&
    &(ii)\quad
    d^e \;=\; \bar{d}^e.
    \tag{$\ast$}
  \end{align}
  The idea behind  $(\ast)$ is to ensure that both populations still make use of the same amount of resources when producing dormant offspring becomes either less ($\ga < 1$) or more ($\ga >1$) resource consuming than producing active offspring.  This can be seen as one particular way of incorporating a reproductive trade-off (another natural one is the introduction of the parameter $\varepsilon >0$ in the BGWPWD from Section \ref{sec:mot-exp}).   Obviously, we recover the notion of fair comparison for $\ga=1$.
  
  To get some intuition on the influence of $\ga$ on the fitness under fair comparison, we provide an example: Indeed, we adjust Example~\ref{exps:strong-advantages} by setting
  \begin{align*}
    M^{\mathrm{res}}(2)
    &\;=\;
    \begin{pmatrix}
      0 & 4\al/\ga\\
      0 & 4/5
    \end{pmatrix},
    &
    M^{\mathrm{pre}}(1)
    &\;=\;
    \begin{pmatrix}
      0 & 4/\ga\\
      0 & 4/5
    \end{pmatrix},
    \\[1ex]      
    M^{\mathrm{sto}}(1)
    &\;=\;
    \begin{pmatrix}
      4/(1+\ga) & 4/(1+\ga)\\
      2/5 & 2/5
    \end{pmatrix},
    &
    M^{\mathrm{sto}}(2)
    &\;=\;
    \begin{pmatrix}
      4\al/(1+\ga) & 4\al/(1+\ga)\\
      2/5 & 2/5
    \end{pmatrix}.
  \end{align*}
  With this, the four processes from the example satisfy the condition for $\ga$-weighted fair comparison, while we still maintain $\det M^{\mathrm{sto}}(1) = \det M^{\mathrm{sto}}(2) = 0$ to obtain an exact result from Theorem~\ref{thm:stochastic-fitness}.  Also, a convex combination $M_q^{\mathrm{cc}}(e)$ yields $\ga$-weighted fair comparison, although not necessarily retaining rank 1 anymore such that $\vp_{\mathrm{cc}}$ requires simulation.

  \begin{figure}[t]
    \includegraphics{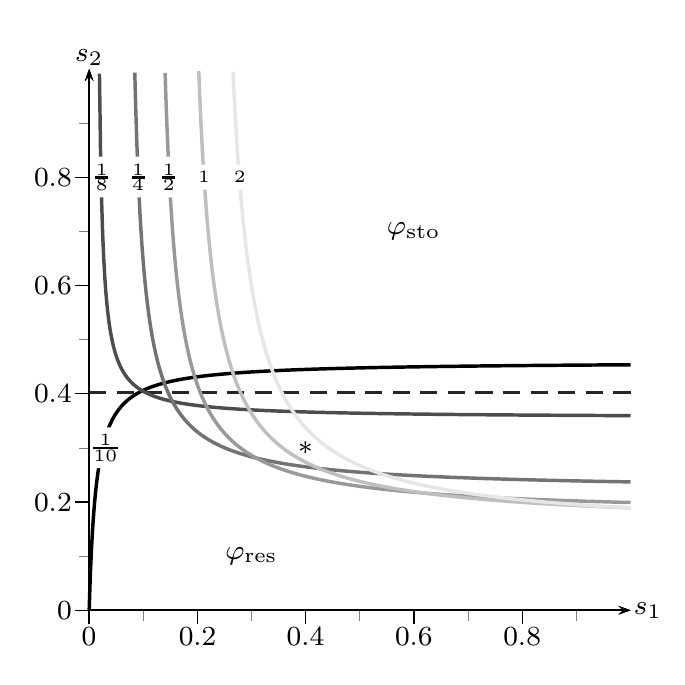}
    \hspace{5mm}
    \includegraphics{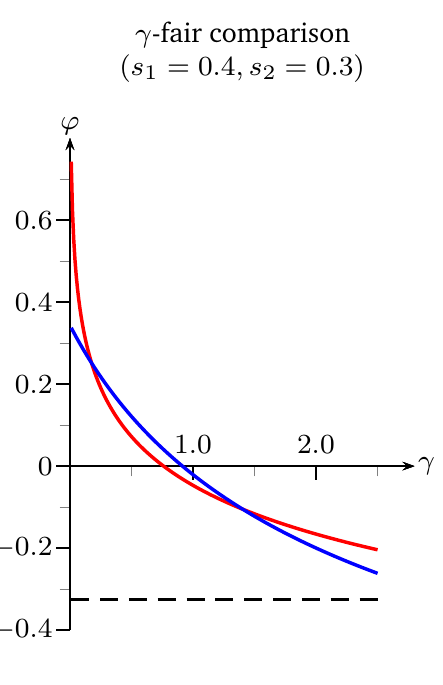}
    \caption{\label{fig:gamma} (Left) Phase diagrams of the maximal Lyapunov exponents $\vp_{\mathrm{res}}$ and $\vp_{\mathrm{sto}}$ of Remark~\ref{rem:weighted-fair-comp} with $\al = 1/20$ under $\ga$-weighted fair comparison for various $\gamma$.  (Right) Lyapunov exponents under $\ga$-weighted fair comparison $s_1=0.4$ and $s_2=0.3$ as functions of $\ga$. Black: $\vp_X$, red: $\vp_{\mathrm{res}}$, blue: $\vp_{\mathrm{sto}}$. Parameters given in Remark~\ref{rem:weighted-fair-comp}.}
  \end{figure}
  Figure~\ref{fig:phase_diagram:fair} illustrates the influence of $\ga$ on the phase diagram in Figure~\ref{fig:phase_diagram} (left).  Here, we see that halving the cost parameter $\ga$ % \in (0,1)$
  largely enhances the advantages of carrying any dormancy trait, where the advantagous region for $\vp_{\mathrm{sto}}$ seems to increase the most.  Naturally, the switching strategies will always dominate the 1-type process as $\ga$ approaches 0, i.e.\ as dormant offspring become very cost-efficient.
  On the other hand, having a cost parameter $\ga > 1$ shifts the landscape in such a way that the 1-type process overtakes the strong-advantage-region from the stochastic switcher.

  Figure~\ref{fig:gamma} (left) shows the phase diagram of stochastic vs.\ responsive switching for various values of $\ga$, where the separatrix in general is given by the equation
  \begin{align}
    \label{eq:separatrix}
    s_2
    \;=\;
    \frac{
      s_1
      \big(
        \log\big(w^2 + w^2 \tfrac{\mDor}{\mAct}\big)
        - \log\big(\al \mAct + w^2 \tfrac{\mDor}{\mAct}\big)
      \big)
    }
    {
      \log\big(\mAct + w^1 \tfrac{\mDor}{\mAct}\big)
      - \log\big( \mAct + \ga \mDor \big)
      - s_1 \log\big(\frac{\al w^1}{\ga w^2}\big)
    }
  \end{align}
  With the parameters specified above, we observe that, for $\ga = 1/9$, the separatrix becomes a constant function with $s_2 = \log(40/29) / \log(20/9) \approx 0.4027$.  Remarkably, this effect leads to parameter regimes where the fitness of the responsive switcher exceeds that of the stochastic switcher if $\ga$ is either small \emph{or} big, while stochastic switching wins for intermediate $\ga$, e.g.\ at the point $(s_1,s_2)=(0.4,0.3)$ marked by $*$.  This particular case is further depicted in Figure~\ref{fig:gamma} (right), where the respective Lyapunov exponents are plotted as functions of $\ga$. (Note that the 1-type-fitness is constant here, since it is not influenced by the cost of dormant offspring.)
\end{remark}
\begin{figure}
  \includegraphics{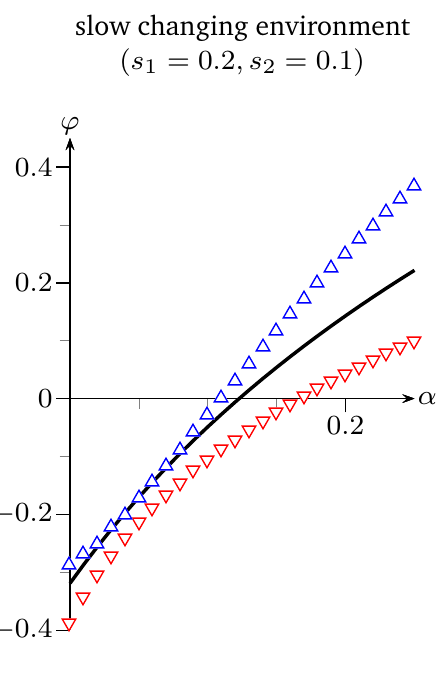}
  \hspace{5mm}
  \includegraphics{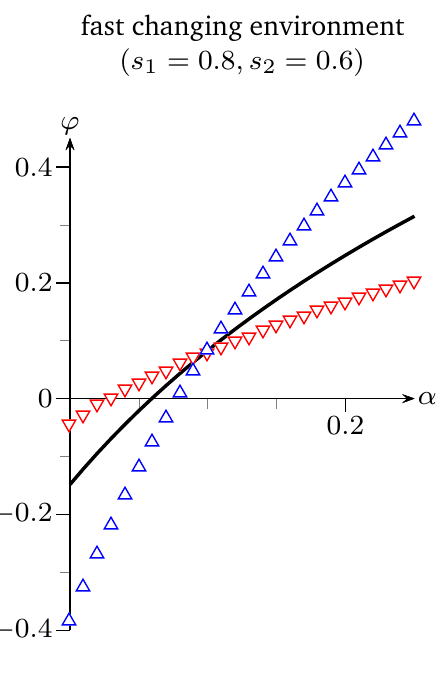}
  \caption{\label{fig:non-zero-det} Left and right: Fair comparison of various stochastic switchers in the regimes of Example~\ref{exps:strong-advantages} (1) and (3) resp., Lyapunov exponents taken as functions of $\al$. Line: $\vp_{\mathrm{sto}}$, $\De,\nabla$: simulated values of $\vp$ for strategies defined by $M_\De,M_\nabla$ resp. from Remark~\ref{rem:non-zero-det}.  %On the right: Lyapunov exponents under $\ga$-weighted fair comparison as functions of $\ga$. Black: $\vp_X$, red: $\vp_{\mathrm{res}}$, blue: $\vp_{\mathrm{sto}}$. Parameters given in Remark~\ref{rem:weighted-fair-comp}.
  }
\end{figure}
\begin{remark}[Non-zero determinant case for mean matrices]\label{rem:non-zero-det}
  Rather than combining the pure strategies, one can also compare different stochastic switching strategies under fair comparison, e.g.\ by choosing mean matrices of non-zero determinant, as illustrated in Figure~\ref{fig:non-zero-det}. Here, we add to the setting of Example~\ref{exps:strong-advantages} two further stochastic switchers with matrices
  \begin{align*}
    M^{\mathrm{sto}}_{\De}(e)
    \;=\;
    \begin{pmatrix}
      13/4 \al^{e-1} & 3/4\al^{e-1}
      \\
      2/5 & 2/5
    \end{pmatrix}
    \quad\text{and}\quad
    M^{\mathrm{sto}}_{\nabla}(e)
    \;=\;
    \begin{pmatrix}
      3/4\al^{e-1} & 13/4\al^{e-1}
      \\
      2/5 & 2/5
    \end{pmatrix}
  \end{align*}
  for $e \in \{1,2\}$.  These satisfy the conditions of fair comparison to the processes in Example~\ref{exps:strong-advantages} while $\det M^{\mathrm{sto}}_{\De}(e) = \al^{e-1} > 0$ and $\det M^{\mathrm{sto}}_{\nabla}(e) = -\al^{e-1} < 0$.

  In contrast to the stochastic switcher given by $M^{\mathrm{sto}}$ in Example~\ref{exps:strong-advantages}, the $\De$-matrices describe a strategy that handles both environments more efficiently by producing more active offspring in good times while also staying dormant for a much longer period.  Hence, this strategy leads to an increase of fitness in rarely changing environments, Figure~\ref{fig:non-zero-det} (left).

  The $\nabla$-matrices however, describe a population that almost entirely produces dormant offspring, while dormant individuals wake up quickly -- this seems comparable to strategies employed by plants.  This gives a strategy that prevails in frequently changing and sufficiently harsh environments as seen in Figure~\ref{fig:non-zero-det} (right). This observation is quite intuitive: In the extreme case where the environment changes every generation and bad environments are sufficiently harsh (e.g.\  winter season), the optimal strategy would be to exclusively produce dormant offspring, which wake up immediately to form the next generation. The setting in Figure~\ref{fig:non-zero-det} (right) resembles an approximation of this extreme case.
  
  Figure~\ref{fig:phase_diagram:non-zero_det} illustrates the influence of the determinant on the phase diagram in Figure~\ref{fig:phase_diagram} (left).
\end{remark}
\begin{figure}[t]
  \includegraphics{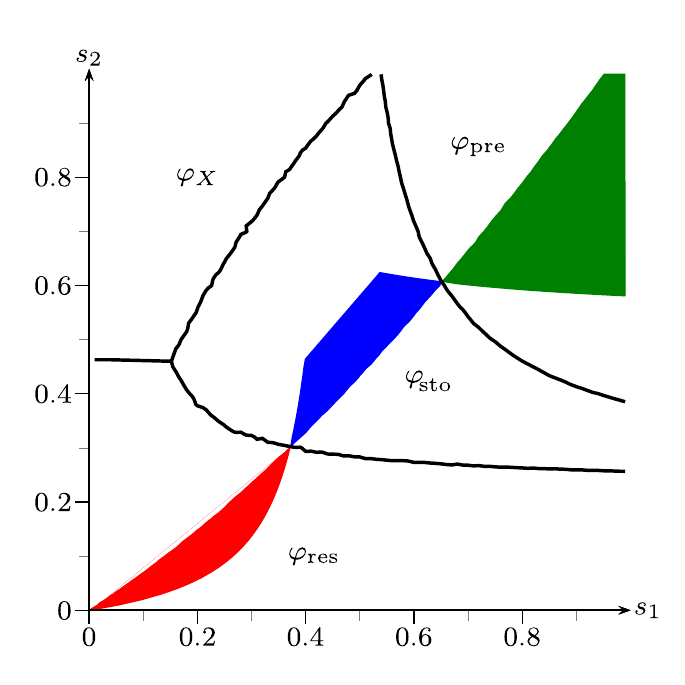}
  \hspace{-0.2cm}
  \includegraphics{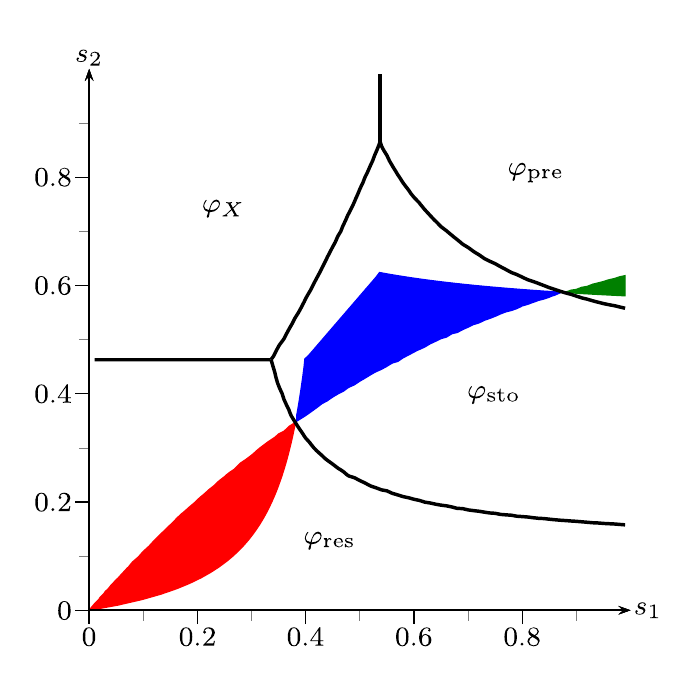}
  \caption{\label{fig:phase_diagram:non-zero_det} Phase diagram of the maximal Lyapunov exponents $\vp_X$, $\vp_{\mathrm{res}}$, $\vp_{\mathrm{sto}}$, $\vp_{\mathrm{pre}}$ of Example~\ref{exps:strong-advantages} and from Remark~\ref{rem:non-zero-det} under fair comparison with $\al = 1/20$. Simulated values of $\vp_{\mathrm{sto}}$ defined by $M_{\De}^{\mathrm{sto}}$ (left) and $M_{\nabla}^{\mathrm{sto}}$ (right). Strong advantage of $\vp_{\mathrm{res}}$ (red), $\vp_{\mathrm{sto}}$ (blue), and $\vp_{\mathrm{pre}}$ (green).}
\end{figure}
Of course, the above examples invite a much larger and systematic study of parameter ranges and switching strategies, but we think that this goes beyond the scope of the present paper, with its focus on mathematical methods. Indeed, in the next section, we try to get some systematic insight into methods for the computation of Lyapunov exponents in our dormancy scenario, including a review of methods that have been used in similar modeling set-ups so far.

\section{Technical results for rank-1 and rank-2 switching strategies}
\label{sec:theory}
In this section, we provide theoretical results for the explicit computation and bounds for maximal Lyapunov exponents related to switching strategies in BGWPDRE. We obtain proofs for Theorems~\ref{thm:responsive-fitness}, \ref{thm:stochastic-fitness} (rank-1 case) and \ref{thm:dirac-lower-bound} (rank-2 case). We also provide additional bounds in the rank-2 case and a short literature review.

\subsection{Rank-1-matrices and exact results}\label{sec:rank-one-case}
Consider a stationary and ergodic sequence $(M_1, M_2, \ldots)$ of non-negative $p \times p$ matrices such that, for any $n \in \bbN$, the rank of $M_n$ is equal to one.  Hence, for any $n \in \bbN$, we can find column vectors $\ell_n, r_n \in [0, \infty)^p$ such that $M_n = \ell_n \cdot r_n^\top$.  Note that in this case, only one eigenvalue of $M_n$ is non-zero and, as a consequence, $\vr(M_n) = \tr M_n$.
\begin{lemma}\label{lem:lyap-exp-det0}
  Let $(M_1, M_2, \ldots)$ be a stationary and ergodic sequence of non-negative $p \times p$ matrices with $M_n = \ell_n \cdot r_n^\top$ for any $n \in \bbN$.  Suppose that $\mean\!\big[\log^+\|\ell_1\|\big] < \infty$ and $\mean\!\big[ \log^+ \|r_1\|\big] < \infty$.  Then, $\prob$-a.s. and in mean, 
  \begin{align}\label{eq:lyap-exp-det0}
    \vp
    \;=\;
    \lim_{n \to \infty} \frac{1}{n} \log \|M_1 \cdot \ldots \cdot M_n\|
    \;=\;
    \mean\!\big[ \log\, \langle r_1, \ell_2 \rangle \big].
  \end{align}
\end{lemma}
\begin{proof}
  First, note that both $\mean\!\big[\log^+ \langle r_1, \ell_2 \rangle\big] < \infty$ and $\mean\!\big[\log^+ \|M_1\|\big] < \infty$.  The latter ensures that the maximal Lyapunov exponent, $\vp$, exists $\prob$-a.s.\ and in mean.  Thus it remains to show that $\vp$ is equal to the expression on the right-hand side of \eqref{eq:lyap-exp-det0}. 

  In order to apply \cite[Theorem~1]{Ki73} which ensures that $\vp$ is finite, we first assume that $\mean\!\big[\log \|M_1 \cdot \ldots \cdot M_n\|\big] \geq -A n$ for some $A \in [0, \infty)$ and all $n \in \bbN$.  This implies that $\log\, \langle r_1, \ell_2 \rangle \in L^1(\prob)$, $\log \|\ell_1\| \in L^1(\prob)$ and $\log\|r_1\| \in L^1(\prob)$.  In particular, $\prob$-a.s., $\langle r_n, \ell_{n+1} \rangle > 0$ for all $n \in \bbN$.  Since
  \begin{align}\label{eq:splitting:rank_1}
    \frac{1}{n} \log \|M_1 \cdot \ldots \cdot M_n\|
    \;=\;
    \frac{1}{n} \sum_{i=1}^{n-1} \log\, \langle r_i, \ell_{i+1} \rangle
    \,+\,
    \frac{1}{n} \log \|\ell_1 \cdot r_n^\top\|
  \end{align}
  we immediately deduce from Birkhoff's ergodic theorem that the first term on the right-hand side of \eqref{eq:splitting:rank_1} converges, $\prob$-a.s.\ and in $L^1(\prob)$, to $\mean\!\big[\log \langle \ell_1, r_2 \rangle\big]$ as $n \to \infty$.  Since $\sup_{n \in \bbN} \mean\!\big[|\log \|\ell_1 \cdot r_n^\top\| |\big] < \infty$, it follows that $\lim_{n \to \infty} \frac{1}{n} \mean\!\big[| \log \|\ell_1 \cdot r_n^\top\||\big] = 0$, and \eqref{eq:lyap-exp-det0} holds in mean.  Moreover, for any $\ve > 0$
  \begin{align*}
    \sum_{n=1}^{\infty}
    \prob\!\big[|\log \|\ell_1 \cdot r_n^\top\| | \geq \ve n\big]
    \;\leq\;
    \frac{2 C}{\ve}
    \Big(
      \mean\!\big[|\log \|\ell_1\||\big]
      \,+\, \mean\!\big[|\log \|r_1\||\big]
    \Big)
    \;<\;
    \infty,
  \end{align*}
  where the constant $C \geq 1$ appearing in the computation above results from the comparison of equivalent matrix norms.  Thus, by using the Borel-Cantelli Lemma we conclude that, $\prob$-a.s., $\lim_{n \to \infty} \frac{1}{n} \log \|\ell_1 \cdot r_n^\top\| = 0$, and \eqref{eq:lyap-exp-det0} follows.
  
  However, if the additional assumption above is violated then we conclude that $\mean\!\big[\log \langle r_1, \ell_2 \rangle\big] = -\infty$.  Thus, by \cite[Theorem~2]{Ki73}, the maximal Lyapunov exponent, $\vp$, as well as the limit of the sum on the right-hand side of \eqref{eq:splitting:rank_1} exists with probability one, and $\vp = -\infty$.  Using that $\sup_{n \in \bbN} \mean\!\big[\log^+ \|\ell_1 \cdot r_n^\top\|\big] < \infty$ concludes the proof.
\end{proof}
\begin{corro}\label{cor:lyap-exp-det0}
  Let $Z$ be a $p$-type BGWPRE in environment $(I_n)$ given in Definition~\ref{def:environment_process}.  Suppose that $\rk M(e) = 1$, $\tr M(e) > 0$ for any $e \in \{1,2\}$, and $\tr ( M(1) \cdot M(2) ) > 0$.  Then, $\prob$-a.s.,
  \begin{align*}
    \vp_Z
    \;=\;
    \tfrac{s_2}{s_1+s_2}\log \tr M(1) + \tfrac{s_1}{s_1+s_2}\log\tr M(2)
    + \tfrac{s_1s_2}{s_1+s_2}
    \log\Big(\frac{\tr(M(1)M(2))}{\tr M(1)\tr M(2)}\Big).
  \end{align*}
\end{corro}
\begin{proof}
  By Lemma~\ref{lem:lyap-exp-det0}, it holds that, $\prob$-a.s.,
  \begin{align*}
    \vp_Z
    \;=\;
    \mean\!\big[ \log\, \langle r(I_0), \ell(I_1) \rangle \big]
    \;=\;  
    \sum_{i,j \in \{1,2\}} \pi_I(i)\, P_I(i,j)
    \log\,\langle r(i), \ell(j) \rangle.
  \end{align*}
  By using that $\langle \ell(e), r(e)\rangle = \tr M(e)$ for any $e \in \{1,2\}$, $\langle r(1), \ell(2)\rangle \langle r(2), \ell(1)\rangle = \tr (M(1) \cdot M(2))$ and $\pi_I(1) P_I(1,2) = \pi_I(2) P_I(2,1)$, the assertion follows.  
\end{proof}
\begin{proof}[Proof of Theorem~\ref{thm:responsive-fitness}, \ref{thm:prescient-fitness} and \ref{thm:stochastic-fitness}]
  This follows directly from Corollary~\ref{cor:lyap-exp-det0}.
\end{proof}
\begin{remark}[Connection to \cite{DMB11}]
  \label{rem:discussion:DMB11} \
  \begin{enumerate}
  \item Similarly to \cite[Propositions~1 and 7]{DMB11}, the responsive switcher can be regarded as a $1$-type BGWPRE process in a more complex random environment, here given by $((I_n,I_{n+1}))_n$ with corresponding offspring means $m_{1,i} = m_i$ and $m_{2,i} = 1-d_i$ for $i \in \{1,2\}$.  With this, Theorem~\ref{thm:responsive-fitness} follows by applying the Ergodic Theorem.
  \item For a given fixed mean offspring per type and environment, say $(m_t(e))_{1\leq t \leq p}$ for $e \in \{1,2\}$, consider a collection of distributions of offspring types -- say $\nu_t(e) \in \bbR_{\geq0}^p$ for $1 \leq t \leq p$, $e \in \{1,2\}$ -- as reproduction strategy.  In \cite{DMB11} the maximal Lyapunov exponent can only be computed explicitly in the so-called \emph{non-hereditary} case, that is, when the distributions of offspring types do not depend on the parent type, $\nu_t(e) = \nu(e)$.  Regarding $m(e) = (m_t(e))_t$ and $\nu(e)$ as column vectors in $\bbR_{\geq0}^p$, the corresponding mean matrices are
    \begin{align*}
      M(e)
      \;=\;
      m(e)\cdot\nu(e)^\top
    \end{align*}
    and thus, of rank 1.  Hence, the case in which \cite{DMB11} obtain exact results for $\vp_Z$ aligns with the case where we do.

    A natural generalization is to give type-$t$-individuals an offspring type distribution depending on $e$ as a convex combination of two distributions, say $\nu(e)$ and $\mu(e)$. This provides a simple example of the \emph{hereditary} case and produces mean matrices of rank at most 2.
  \end{enumerate}
\end{remark}
\subsection{Rank-2-matrices and bounds}\label{sec:rank-2-case}
This section provides an overview of several bounds for upper Lyapunov exponents that can be found in the literature. A comparison of these with respect to the application to stochastic switching will be given as well as remarks on potential improvements.

In the sequel, consider a stationary and ergodic sequence $(M_1, M_2, \ldots)$ of non-negative $p \times p$ matrices such that, for any $n \in \bbN$, the rank of $M_n$ is at most two, i.e.\ we can find column vectors $\ell_n^i, r_n^i \in [0, \infty)^p$, $i \in \{1,2\}$ such that 
\begin{align*}
  M_n \;=\; \sum_{i=1}^2 \ell_n^i \cdot (r_n^i)^\top
\end{align*}
for any $i \in \bbN$. Further, for any $i \in \bbN$, we denote by $A_{n,n+1}$ a non-negative $2 \times 2$ matrix that is defined by
\begin{align}
  A_{n,n+1}
  \;\ldef\;
  \begin{pmatrix}
    \langle r_n^1, \ell_{n+1}^1 \rangle & \langle r_n^1, \ell_{n+1}^2 \rangle
    \\[.5ex]
    \langle r_n^2, \ell_{n+1}^1 \rangle & \langle r_n^2, \ell_{n+1}^2 \rangle
  \end{pmatrix}.
\end{align}
Note that the sequence $(A_{1,2}, A_{2,3}, \ldots)$ is as well stationary and ergodic.
\begin{remark}\label{rem:matrix-decomp}
  There are several ways to decompose a non-negative $2\times2$-matrix into the sum of two products of non-negative vectors, e.g.\ for any $a,b,c,d\geq0$ and $ab>0$ it holds
  \begin{align}
    A
    \;\ldef\;
    \begin{pmatrix}
      a&b\\c&d
    \end{pmatrix}
    &\;=\;
    \begin{pmatrix} 1 \\ 0 \end{pmatrix}
    \cdot
    \begin{pmatrix} a & b \end{pmatrix}
    + \begin{pmatrix} 0 \\ 1 \end{pmatrix}
    \cdot
    \begin{pmatrix} c & d \end{pmatrix}
    \label{eq:vec-decom-row}\\[1em]
    &\;=\;
    \begin{pmatrix} a \\ c \end{pmatrix}
    \cdot
    \begin{pmatrix} 1 & 0 \end{pmatrix}
    + \begin{pmatrix} b \\ d \end{pmatrix}
    \cdot
    \begin{pmatrix} 0 & 1 \end{pmatrix}
    \label{eq:vec-decom-column}\\[1em]
    &\;=\;
    \begin{pmatrix} a \\ c \end{pmatrix}
    \cdot
    \begin{pmatrix} 1 & \frac ba \end{pmatrix}
    + \begin{pmatrix} 0 \\ 1 \end{pmatrix}
    \cdot
    \begin{pmatrix} 0 & \frac{\det A}a \end{pmatrix}
    \label{eq:vec-decom-pos-det}\\[1em]
    &\;=\;
    \begin{pmatrix} b \\ d \end{pmatrix}
    \cdot
    \begin{pmatrix} \tfrac ab & 1 \end{pmatrix}
    + \begin{pmatrix} 0 \\ 1 \end{pmatrix}
    \cdot
    \begin{pmatrix} \frac{-\det A}b & 0 \end{pmatrix},
    \label{eq:vec-decom-neg-det}
  \end{align}
  where the entries in \eqref{eq:vec-decom-row} and \eqref{eq:vec-decom-column} are always non-negative, in \eqref{eq:vec-decom-pos-det} they are non-negative if $\det A\geq0$ and in the last if $\det A\leq 0$. Notably, \eqref{eq:vec-decom-column} corresponds to writing $M(e)$ as convex combination of responsive and prescient switchers as indicated in Remark~\ref{rem:convex-combination}.
\end{remark}
\begin{lemma}\label{lem:lyap-exp-rank2}
  Let $(M_1, M_2, \ldots)$ be a stationary and ergodic sequence of non-negative $p \times p$ matrices with $M_n = \sum_{i=1}^2 \ell_n^i \cdot (r_n^i)^\top$ for any $n \in \bbN$ and $\log \|\ell_1^i\|, \log \|r_1^i\| \in L^1(\prob)$ for any $i \in \{1,2\}$.  Then, $\prob$-a.s. and in mean, 
  \begin{align}\label{eq:lyap-exp-rank2}
    \vp
    \;=\;
    \lim_{n \to \infty} \frac{1}{n} \log \|M_1 \cdot \ldots \cdot M_n\|
    \;=\;
    \lim_{n \to \infty}
    \frac{1}{n} \log \|A_{1,2} \cdot \ldots \cdot A_{n-1,n}\|.
  \end{align}
\end{lemma}
\begin{proof}
  First, by an elementary computation, we get that $\mean\!\big[\log^+ \|M_1\|\big] < \infty$ and $\mean\!\big[\log^+ \|A_{1,2}\|\big] < \infty$.  Thus, \cite[Theorem~6]{Ki73} implies that
  \begin{align*}
    \lim_{n \to \infty} \frac{1}{n} \log \|M_1 \cdot \ldots \cdot M_n\|
    \qquad \text{and} \qquad
    \lim_{n \to \infty}
    \frac{1}{n} \log \|A_{1,2} \cdot \ldots \cdot A_{n-1,n}\|
  \end{align*}
  exist $\prob$-a.s.\ and in mean.  Thus, we are left with showing that both limits coincide.  Recall the limit does not depend on the chosen matrix norm.  Thus, for the matrix norm $\|B\| \ldef \sum_{i,j=1}^p |B^{i,j}|$ we obtain
  \begin{align*}
    \| M_1 \cdot \ldots \cdot M_n \|
    \;=\;
    \sum_{i,j=1}^2 \big(A_{1,2} \cdot \ldots \cdot A_{n-1,n}\big)^{i,j}\,
    \|\ell_1^i\|_1\, \|r_n^j\|_1.
  \end{align*}
  By setting $R_n \ldef \sum_{i=1}^2 (|\log \|\ell_1^i\|_1| + |\log \|r_n^i\|_1 |)$ for any $n \in \bbN$, it follows that
  \begin{align*}
    -\frac{1}{n} R_n
    \;\leq\;
    \frac{1}{n} \log \|M_1 \cdot \ldots \cdot M_n\|
    \,-\,
    \frac{1}{n} \log \|A_{1,2} \cdot \ldots \cdot A_{n-1,n} \|
    \;\leq\;
    \frac{1}{n} R_n.
  \end{align*}
  Thus, by using the same argument as in the proof of Lemma~\ref{lem:lyap-exp-det0}, we obtain that, $\prob$-a.s.\ and in mean, $\lim_{n \to \infty} \frac{1}{n} R_n = 0$, which concludes the proof.
\end{proof}
Next, we focus on establishing various bounds for the maximal Lyapunov exponent for the resulting product of $2 \times 2$ matrices.
\begin{prop}\label{prop:bounds}
  Let $(I_n)$ be a stationary and ergodic Markov chain with values in $\Om' = \{1,2\}$ as given in Definition~\ref{def:environment_process}, and $A\!: \Om' \times \Om' \to [0, \infty)^{2 \times 2}$ such that $\mean[|\log\|A(I_0, I_1)\||] < \infty$.  Then the following statements hold:
  \begin{enumerate}[(a)]
  \item For $\la\!: \Om' \times \Om' \to (0, \infty)$ set $A^*_{n, n+1} \equiv A^*(I_n, I_{n+1}) \ldef A(I_n, I_{n+1}) / \la(I_n, I_{n+1})$.  Then, $\prob$-a.s. and in mean,
    \begin{align*}
      \lim_{n \to \infty}
      \frac{1}{n} \log \big\|A_{1,2} \cdot \ldots \cdot A_{n-1,n}\big\|
      \;\leq\;
      \mean[\log \la(I_0,I_1)] \,+\, \log \vr(\widehat{A}^*),
%      \label{eq:jensen-upper-bound}
    \end{align*}
    where $\vr(\widehat{A}^*)$ denotes the spectral radius of the $(4 \times 4)$-matrix $\widehat{A}^*$ which is given by
    \begin{align*}
      \widehat{A}^*
      \;\ldef\;
      \begin{pmatrix}
        (1-s_1) A^*(1,1) & s_1 A^*(1,2) \\
        s_2 A^*(2,1) & (1-s_2) A^*(2,2)
      \end{pmatrix}.
%      \label{eq:jensen-lower-bound}
    \end{align*}
  \item For $n\geq1$ denote by $\mathcal D_n$ the set of probability density functions on $\{1,2\}^n$. Then, $\prob$-a.s. and in mean,
    \begin{align*}
      \lim_{n \to \infty}
      \frac{1}{n} \log \big\|A_{1,2} \cdot \ldots \cdot A_{n-1,n}\big\|
      \;\geq\;
      \limsup_{n\to\infty}\sup_{\nu \in \cD_n} \frac{1}{n}
      \bigg(
        \sum_{k=1}^{n-1} \bfE_\nu[X_k] + H(\nu)
      \bigg),
    \end{align*}
    where $X_k = \log A(I_k,I_{k+1})^{\al_k,\al_{k+1}}$, $\bfE_\nu$ denotes integration by $\al\in\{1,2\}^n$ with respect to $\nu$ and
    \begin{align*}
      H(\nu)
      \;=\;
      -\sum_{\al \in \{1,2\}^n} \nu(\al)\log\nu(\al)
    \end{align*}
    the entropy of $\nu$.
  \end{enumerate}
\end{prop}
\begin{proof}
  \emph{(a)} First, by the ergodic theorem, we have that, $\prob$-a.s. and in mean, 
  \begin{align*}
    &\lim_{n \to \infty}
    \frac{1}{n} \log \big\|A_{1,2} \cdot \ldots \cdot A_{n-1,n}\big\|
    \\[.5ex]
    &\mspace{36mu}=\;
    \mean\!\big[\log \la(I_0, I_1)\big]
    \,+\,
    \lim_{n \to \infty}
    \frac{1}{n}
    \mean\!\Big[\log \big\|A^*_{1,2} \cdot \ldots \cdot A^*_{n-1,n}\big\|\Big].
  \end{align*}
  Moreover, it is well known that an upper bound for the maximal Lyapunov exponent of the stationary and ergodic sequence $(A^*_{1,2}, A^*_{2,3}, \ldots)$ follows immediately from Jensen's inequality.  Indeed,
  \begin{align*}
    \mean\!\Big[\log \big\|A^*_{1,2} \cdot \ldots \cdot A^*_{n-1,n}\big\|\Big]  
    %&\mspace{36mu}\leq\;
    &\;\leq\;
    \log  \mean\!\Big[\big\|A^*_{1,2} \cdot \ldots \cdot A^*_{n-1,n}\big\|\Big]
    \\[.5ex]
    &\;=\;  
    \log (\pi_I \otimes \mathbf{1}_2)^\top (\widehat{A}^*)^{n-1} (\mathbf{1}_4),
  \end{align*}
  where $\mathbf{1}_k \ldef (1, \ldots, 1) \in \bbR^k$.  Since
  \begin{align*}
    \log \min_{i \in \{1,2\}} \pi_I(i) + \log \|(\widehat{A}^*)^{n-1}\|
    \;\leq\;
    \log (\pi_I \otimes \mathbf{1}_2)^\top (\widehat{A}^*)^{n-1} (\mathbf{1}_4)
    \;\leq\;
    \log \|(\widehat{A}^*)^{n-1}\|,
  \end{align*}
  where we used the matrix norm $\|B\| = \sum_{i,j} |B^{i,j}|$, $B \in \bbR^{2 \times 2}$, the assertion follows from \cite[Corollary~5.6.14]{HJ90}. 

  \emph{(b)} Note that, for any $\nu \in \cD_n$,
  \begin{align*}
    \|A_{1,2} \cdot \ldots \cdot A_{n-1,n}\|
    &\;=\;
    \sum_{\al \in \{1,2\}^n} \prod_{k=1}^{n-1} A_{k, k+1}^{\al_k, \al_{k+1}}
    \\
    &\;\geq\;
    \sum_{\al:\,\nu(\al)>0}
    \nu(\al)
    \exp\!\bigg(
      \sum_{k=1}^{n-1}\log A_{k,k+1}^{\al_k,\al_{k+1}} - \log\nu(\al)
    \bigg)
  \end{align*}
  The result follows from Jensen's inequality, taking supremum and limit superior. Note that the right-hand side converges almost surely and hence in mean by monotone convergence.
\end{proof}
\begin{proof}[Proof of Theorem~\ref{thm:dirac-lower-bound}]
  Since $w$ and $d$ do not depend on $e$, it holds $\det M(2) = \al \det M(1)$. Notably, when $\det M(1)=0$, this lower bound equates to the result from Theorem~\ref{thm:stochastic-fitness}. Hence, in what follows we assume $\det M(1)\neq0$.

  For $\det M(1) > 0$, using representation \eqref{eq:vec-decom-pos-det} we obtain
  \begin{align*}
    A(i,j)
    \;=\;
    \begin{pmatrix}
      \al^{j-1}\mAct+\frac{w\mDor}{\mAct} & \frac{\mDor}{\mAct}
      \\[.5ex]
      \frac{w\det M(1)}{\mAct} & \frac{\det M(1)}{\mAct}
    \end{pmatrix},
    % \;=\;
    % \begin{pmatrix}
    %   \tr M(I_{n+1}) - \frac{\det M(1)}{\mAct} & \frac{\mDor}{\mAct}
    %   \\[.5ex]
    %   \frac{w\det M(1)}{\mAct} & \frac{\det M(1)}{\mAct}
    % \end{pmatrix}
  \end{align*}
  where $A(i,j)_{1,1}=\tr M(j)-\frac{\det M(1)}{\mAct}$.
  On the other hand, if $\det M(1)<0$ and we use \eqref{eq:vec-decom-neg-det},
  \begin{align*}
    A(i,j)
     &= \begin{pmatrix}
          \tr M(j) & 1\\
          -\det M(j) & 0
        \end{pmatrix}.
  \end{align*}
  Hence, in both cases it holds $A(i,j)_{1,1}=\tr M(j)-(\det M(1)/\mAct)^+$ and Proposition~\ref{prop:bounds}\emph{(b)} concludes the proof by considering $\nu_n=\delta_{\{1\}^n}\in\mathcal D_n$ and applying the Ergodic Theorem.
\end{proof}
\begin{remark}[Further bounds on the maximal Lyapunov exponent]\label{rem:further-bounds}
  Let us consider the stochastic switching model with mean matrices $M(1)$ and $M(2)$.
  \begin{enumerate}
    \item If $w^1 \leq w^2$ and $w^1+d^1 \leq w^2+d^2$, then $M(2)^{i,j} \leq M(1)^{i,j}$ for any $i,j \in \{1,2\}$. In particular, $\|M(2)^n\| \leq \|M_1 \cdot \ldots \cdot M_n\| \leq \|M(1)^n\|$.  Thus, by \cite[Corollary~5.6.14]{HJ90}, we obtain that, $\prob$-a.s. and in mean,
    \begin{align*}
      \ln \vr(M(2))
      % \;=\;
      % \lim_{n \to \infty} \frac{1}{n} \ln \|M(2)^n\|
      \;\leq\;
      \vp_Z
      \;\leq\;
      % \lim_{n \to \infty} \frac{1}{n} \ln \|M(1)^n\|
      % \;=\;
      \ln \vr(M(1)).
    \end{align*}
    Note that such kind of worst-case/best-case estimate has also been obtained in \cite[Proposition~10]{MS08}.  In particular, this bound does not take into account the lengths of the environmental phases given by $s_e^{-1}$ and hence cannot capture the effects illustrated in Remark~\ref{rem:interpr-advs}.
  \item In view of Remark~\ref{rem:simple:bounds} any sub-multiplicative function $\|\cdot\|\!: \bbR_{\geq}^{2 \times 2} \to (0, \infty)$ yields that $\vp_Z \leq \mean[\log\|M(I_0)\|]$.  Likewise, for any super-multiplicative function $f\!: \bbR_{\geq}^{2 \times 2} \to (0, \infty)$ we obtain that $\vp_Z \geq \mean[\log f(M(I_0))]$.  Examples of super-multiplicative functions are the minimal column and row sums, respectively, any diagonal element, or the permanent of a matrix $A$.

    For an slightly improved upper bound note that for any sub-multiplicative matrix norm $\|\cdot\|$
    \begin{align*}
      \Big\|\prod_{k=1}^nM_k\Big\|
      \;\leq\;
      \prod_{k=1}^n\|M_k\| \cdot
      \prod_{k=1}^{n-1}\Big(
        \frac{\|M(1)M(2)\|}{\|M(1)\|\|M(2)\|}
      \Big)^{\indicator_{I_k=1,I_{k+1}=2}},
    \end{align*}
    which takes into account the effects of one type of environmental change.   Hence, we obtain that $\vp_Z \leq \mean[\log\|M(I_0)\|] + \Psi$, where
    \begin{align*}
      \Psi
      \;=\;
      \frac{s_1 s_2}{s_1+s_2}
      \log\Big(
        \frac{\min\{\|M(1)M(2)\|,\|M(2)M(1)\|\}}{\|M(1)\|\|M(2)\|}
      \Big)
      \;\leq\;
      0.
    \end{align*}
    As one can see in Figure~\ref{fig:bounds}, for small $\al$ in some cases this can give a better upper bound than the one from \cite{HL16}.
  \end{enumerate}
\end{remark}
A more evolved approach is to choose a sequence $(\nu_n)$ such that $(\al_k)_k$ can be interpreted as path of a Markov chain. Combining this ansatz with Markov chain limit results leads to

\begin{corro}\label{cor:lower-bound-Markov-chain}
  For $i,j,y\in\{1,2\}$ let $\mu_{ijy}\in[0,1]$, such that the stochastic matrix
  $Q$ defined as
  \begin{align*}
   \bordermatrix{
     ~ & 11               & 12                   & 21               & 22 \cr
    11 & (1-s_1)\mu_{111} & (1-s_1)(1-\mu_{111}) & s_1\mu_{121} & s_1(1-\mu_{121}) \cr
    12 & (1-s_1)\mu_{112} & (1-s_1)(1-\mu_{112}) & s_1\mu_{122} & s_1(1-\mu_{122}) \cr
    21 & s_2\mu_{211} & s_2(1-\mu_{211}) & (1-s_2)\mu_{221} & (1-s_2)(1-\mu_{221}) \cr
    22 & s_2\mu_{212} & s_2(1-\mu_{212}) & (1-s_2)\mu_{222} & (1-s_2)(1-\mu_{222}) \cr
     }
  \end{align*}
  is irreducible and aperiodic, and denote by $q$ its stationary distribution.  Then,
  \begin{align*}
    \varphi_Z
    \;\geq\;
    \sum_{i,j,y,z\in\{1,2\}} q_{iy} Q^{iy,jz}
    \Big(\log A(i,j)^{y,z} + h(\mu_{ijy})\Big),
  \end{align*}
  where $h(x)=0$ if $x\in\{0,1\}$ and $h(x)=-x\log x-(1-x)\log(1-x)$ otherwise.
\end{corro}
Let us point out that this result can also be deduced directly from \cite[Theorem~4.3]{ADG94}, where the authors, based on concepts from equilibrium statistical mechanics, established a variational characterization of the maximal Lyapunov exponent for general ergodic sequences of positive matrices statisfying certain integrability conditions.  Nevertheless, for the sake of being self-contained we provide a proof of Corollary~\ref{cor:lower-bound-Markov-chain} at the end of this section.  A similar upper bound has been derived by Gharavi and Anantharam \cite{GA05}.

Note that Corollary~\ref{cor:lower-bound-Markov-chain} in this special case offers an analytical approach for finding an over-all reliable lower bound by adjusting the eight $\mu$-parameters.  Additionally, it provides a way to give approximate uniform lower bounds (e.g.\ in $\al$, see Figure~\ref{fig:bounds} left and mid).
\begin{remark}[Connection to \cite{HL16} and \cite{KL05}]\label{rem:connection-to-HL16}\
  \begin{enumerate}
  \item Using equation \eqref{eq:vec-decom-row} for both mean matrices of the stochastic switcher yields $A(i,j)=M(i)$.  Hence, letting $\la(e,\cdot) = \vr(e)$, Proposition~\ref{prop:bounds}-(a) gives the same upper bound as in \cite[Theorem~2]{HL16}.  Changing the values of $\la$ allows to influence the loss from the estimation by Jensen's inequality.  Hence, the freedom to choose $\la$ offers potential improvement for this upper bound.
  \item Analogous to Corollary~\ref{cor:lower-bound-Markov-chain} the method in \cite{HL16} is based on constructing $\nu_n$ via transition matrices of the form
    \begin{align*}
      \Th(e)
      \;\ldef\;
      \diag(v(e))^{-1} \cdot \frac{M(e)}{\vr(e)} \cdot \diag(v(e)),
    \end{align*}
    where $v(e)$ denotes the respective and suitably normalized right-eigenvec\-tors of $M(e)$.  In fact, the lower bound in \cite[Theorem~3]{HL16} can be achieved from Corollary~\ref{cor:lower-bound-Markov-chain} by choosing $\mu_{ijy} = \Th(i)_{y1}$ and, as above, $A(i,j)=M(i)$ by decomposition \eqref{eq:vec-decom-row}.  Then,
    \begin{align*}
      Q
      \;=\;
      \begin{pmatrix}
        \Th(1) & 0 \\
        0 & \Th(2)
      \end{pmatrix}
      \cdot
      \bigg( P_I \otimes \begin{pmatrix} 1 & 0 \\ 0 & 1 \end{pmatrix} \bigg)
    \end{align*}
    as well as
    \begin{align*}
      &\sum_{j,z \in \{1,2\}}
      Q^{iy,jz}\Big(\log A(i,j)^{y,z} + h(\mu_{ijy}) \Big)
      \\[.5ex]
      &\mspace{36mu}=\;
      \sum_{j=1}^2 P_I^{ij}
      \Big(
        \Th(i)^{y1} \log M(i)^{y1} + \Th(i)^{y2}\log M(i)^{y2}
        + h(\Th(i)^{y1})
      \Big)
      \\[.5ex]
      &\mspace{36mu}=\;
      \log \vr(i) + \log v(i)_y - \Th(i)^{y1} \log v(i)_1
      - \Th(i)^{y2}\log v(i)_2,
    \end{align*}
    such that the lower bound in Corollary~\ref{cor:lower-bound-Markov-chain} becomes
    \begin{align}
      \mean[\log\vr(0)]
      + q
      \bigg(
        \mathbb{I}_4 - \begin{pmatrix} \Th(1) & 0 \\ 0 & \Th(2) \end{pmatrix}
      \bigg)
      \begin{pmatrix}
        \log v(1)_1 \\ \log v(1)_2 \\ \log v(2)_1 \\ \log v(2)_2
      \end{pmatrix},\label{eq:HL-lower-bound}
    \end{align}
    which illustrates the connection. Figure~\ref{fig:bounds} demonstrates that there are choices for the parameters $\mu$ that can be made to improve the lower bound from \cite{HL16}, especially for small $\al$.  This particular choice of transition matrices $\Theta(e)$ defines a Markov chain $(Y_k)$ closely related to the so-called \emph{retrospective process} (cf.\ \cite[Chapter~3]{Wa14}).
  \item Kussel and Leibler \cite{KL05} approximate the maximal Lyapunov exponent under a \emph{slow environment condition}.  The soundness of this approximation can be verified by the above discussed bounds of \cite{HL16}: As $s_1, s_2 \to 0$ and $s_1/s_2 \to \tau>0$, $P_I$ gets close to $\mathbb{I}_2$ and hence, the second addend in \eqref{eq:HL-lower-bound} approaches $q(\mathbb{I}_4-Q)\log v=0$.  On the other hand, the $\widehat{A}^*$-matrix obtained in Remark~\ref{rem:connection-to-HL16}(1) tends to
    \begin{align*}
      \begin{pmatrix}
        M(1)/\vr(1) & 0 \\ 0 & M(2)/\vr(2)
      \end{pmatrix},
    \end{align*}
    such that $\vr(\widehat A^*) \to 1$.  Thus, both bounds approach $\mean[\log \vr(0)]$ and so does $\vp_Z$.
  \end{enumerate}
\end{remark}
\begin{figure}
  \includegraphics{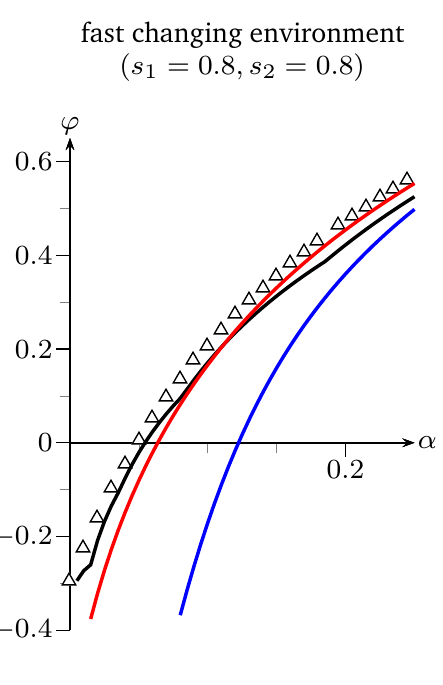}
  \hspace{0mm}
  \includegraphics{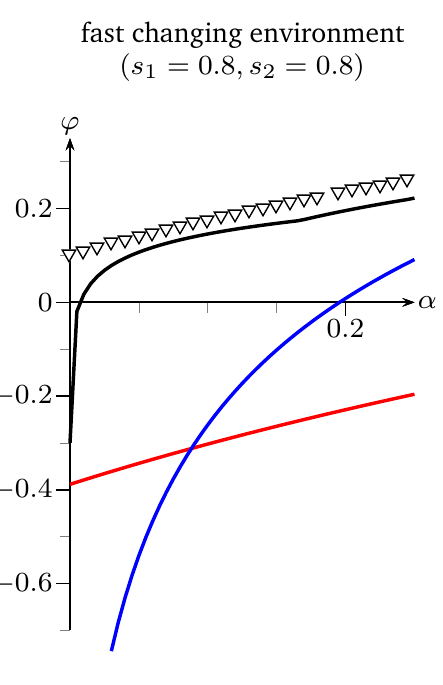}
  \hspace{0mm}
  \includegraphics{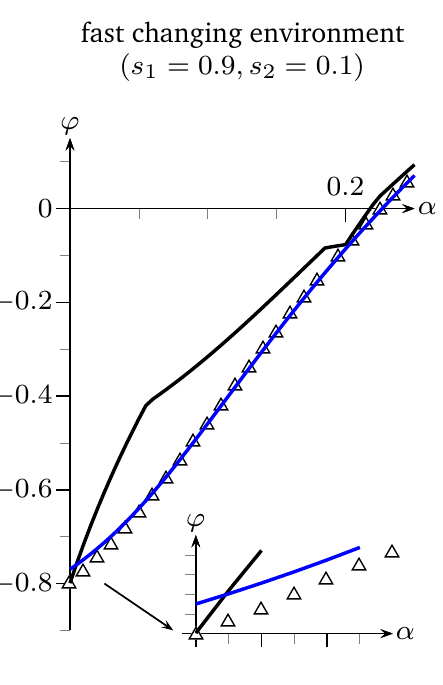}
  \caption{\label{fig:bounds}
    Left and mid: Comparison of lower bounds of $\De$- and $\nabla$-strategy resp.\ in setting of Figure~\ref{fig:non-zero-det} (right) -- red: Theorem~\ref{thm:dirac-lower-bound}, black: Corollary~\ref{cor:lower-bound-Markov-chain} maximized over 1000 random choices of the $\mu$-parameters,  blue: \cite[Theorem~3]{HL16}, $\De$ and $\nabla$: approximation via simulation.
    On the right: Comparison of upper bounds -- %red: Proposition~\ref{prop:bounds}\emph{(a)} with \eqref{eq:vec-decom-row} and $\lambda(e,\cdot)=\tr M(e)$,
    black: improved norm bound from Remark~\ref{rem:further-bounds}-(2) with respect to $\|\cdot\|_1$, blue: \cite[Theorem~2]{HL16}, $\De$: approximation via simulation  -- each again as functions of $\al$ via fair comparison.
  }
\end{figure}

\begin{proof}[Proof of Corollary~\ref{cor:lower-bound-Markov-chain}]
  Denote by $\ga_i = \frac{q_{i1}}{q_{11}+q_{12}}$ and by $(Y_k)_{k \geq 1}$ random variables on $\{1,2\}$ holding $\prob[Y_1=1 \mid I_1] = \ga_{I_1}$ and
  \begin{align*}
    \prob[Y_{n+1} = 1 \mid Y_n, I_n, I_{n+1}] \;=\; \mu_{I_n,I_{n+1},Y_n}
  \end{align*}
  for all $n \geq 1$.  Then, $((I_n,Y_n))_n$ is a time-homogeneous stationary Markov chain with transition matrix $Q$.  Furthermore, $(\bbY_n)_n$ with $\mathbb{Y}_n \ldef ((I_n,Y_n),(I_{n+1},Y_{n+1}))$ is a homogeneous Markov chain with stationary distribution $q^{(2)}$ given by
  \begin{align*}
    q^{(2)}_{ab} \;=\; q_a\cdot Q^{ab} \;=\; (\diag(q)\cdot Q)^{ab}
  \end{align*}
  for $a, b \in \{11, 12, 21, 22\}$.  Now, $X_n = f(\bbY_n)$ with $f(i,y,j,z) = \log A(i,j)_{y,z}$.  Letting $\nu_n^I$ the distribution of $(Y_k)_{1 \leq k \leq n}$ conditional on $(I_k)$, it follows by stationarity of the $\bbY_k$ and ergodicity that
  \begin{align*}
    \frac{1}{n} \sum_{k=1}^{n-1}\bfE_{\nu_n^I}[X_k]
    \;=\;
    \frac1n\sum_{k=1}^{n-1}\mean[f(\mathbb Y_k)\mid I_k,I_{k+1}]
    \;\xrightarrow{n \to \infty}\;
    \mean_{q^{(2)}}\!\big[ f(I_1, Y_1, I_2, Y_2)\big]
  \end{align*}
  amounting to the first addend of the lower bound. On the other hand, using the Markov property,
  \begin{align*}
    \frac{1}{n} H(\nu_n^I)
    &\;\ldef\;
    -\frac{1}{n} \sum_{\al \in \{0,1\}^n} \nu_n^I(\al) \log \nu_n^I(\al)
    \;=\;
    -\frac{1}{n}
    \mean\!\big[\log\nu_n^I(Y_1, \ldots, Y_n) \mid(I_k)_{k \leq n}\big]
    \\[.5ex]
    &\;=\;
    \frac{1}{n}
    \sum_{k=1}^{n-1} \mean\!\big[ h(\mu_{I_kI_{k+1}Y_k}) \mid I_k,I_{k+1} \big]
    + O(1/n).
  \end{align*}
  Hence, by stationarity of $(\bbY_k)$ and the Ergodic Theorem it follows
  \begin{align*}
    \frac{1}{n} H(\nu_n^I)
    \;\xrightarrow{n\to\infty}\;
    \mean_{q^{(2)}}\!\big[h(\mu_{I_1I_2Y_1})\big]
  \end{align*}
  and the corollary holds by Proposition~\ref{prop:bounds}-\emph{(b)}.
\end{proof}

\section{Discussion and outlook}
\label{sec:discussion}
\subsection{Discussion} 
In the previous sections, we investigated 2-type branching process models in random environments that are tailored to the modeling of populations with seed banks comprised of dormant individuals. These models incorporate different switching strategies between active and dormant states (including responsive and spontaneous switching, and mixtures of these), type-specific  relative reproductive costs, and stationary and ergodic environments switching randomly between healthy and harsh states of variable severity. They are in a tradition of earlier (deterministic and stochastic) models from population biology incorporating switches, in particular those of Malik and Smith \cite{MS08} for dormancy described by deterministic dynamical systems, and of Dombry, Mazza and Bansaye \cite{DMB11}, describing phenotypic switches by branching processes in random environments.  We will now discuss our models and results, including their distinctive features and novelties, from a somewhat elevated perspective,  and try to put them in the context of this earlier work, thus contributing to the conceptual understanding of how switching strategies, distributional properties of environments, reproductive costs and stochastic vs.\ deterministic modeling affect the (potential) fitness benefits of dormancy traits in unstable environmental conditions.

\subsection*{`Take-home messages' and relation to the `rules of thumb' of Malik and Smith}
In \cite{MS08}, the authors provide rigorous results in a deterministic dynamical-systems based set-up. Although their modeling approach is different from ours, and they derive most of their results for deterministic and periodically switching environments, it is still instructive to compare them, since there are important similarities, but also distinctions and  novelties obtained in our framework. In their discussion section, Malik and Smith provide several `rules of thumb' summarizing their findings. We try to formulate a corresponding set of such rules; however, one should be careful with these necessarily vague statements - in doubt one always needs to come back to the exact mathematical results, or carry out additional simulations to cover the concrete scenario in question. Note that the following statements refer to our `fair comparison' assumption ($\ga=1$).

\begin{quote}
   \emph{(1)\hspace{1em} Each switching strategy is %Both the responsive and the stochastic switchers are
   more fit than the `sleepless' population when i) good times are rare and ii) bad times are sufficiently harsh. Otherwise, the sleepless case has a (potentially strong) fitness advantage.}
\end{quote}
This was essentially also observed in \cite{MS08}. Here, i)  corresponds in our setting to the condition $s_2\ll s_1$. In this case indeed we see that $\vp_X$ is small -- cf.\ Figure~\ref{fig:phase_diagram}. However, one should  note that this observation also depends on the severity $\al$ of the `harsh' environmental state, as we observe that $\vp_X$ will eventually always dominate when $\al$ increases to 1, thus explaining the additional condition ii).  In fact, we can strengthen this rule by adding that both the responsive or the stochastic switcher, under suitable reproductive parameters and under fair comparison, can even be exclusively super-critical (``strong fitness advantage''), if the ratio between good times and bad times is sufficiently balanced, cf.\ the colored areas of Figure~\ref{fig:phase_diagram}. From now on, we will always assume a sufficiently severe harsh environment (that is, $\al$ is sufficiently small.

Note that for the limit  $\al \downarrow 0$ the horizontal line in Fig.~\ref{fig:phase_diagram} tends to a horizontal line through the point $1$, the slope of the linear function through the origin tends to $\infty$, and the separatrix between the responsive and spontaneous switcher tends to the function $[0,1]  \ni s_1 \mapsto 1-\mathbbm{1}_{(0,1]}(s_1)$. In this case, the stochastic switcher  completely dominates the diagram, but has no strong fitness advantage.

\begin{quote}
  \emph{(2)\hspace{1em} The responsive switcher is more fit than the stochastic switcher when environmental states change rarely, i.e. when $s_1\cdot s_2$ is small.}
\end{quote}
Note that a similar rule has been stated in \cite{KL05}.
The corresponding rule in \cite{MS08} is that the ``responsive switcher is more fit than the stochastic switcher when either good times are very rare or are very common''. In our case, this would correspond to  the condition that either $s_1 \ll s_2$ or $s_2 \ll s_1$, suggesting that the point $(s_1,s_2)$ lies below the graph of a suitable `hyperbola'.  Indeed, we are able to compute the exact boundary of the region where $\vp_{\mathrm{res}}>\vp_{\mathrm{sto}}$, which is given by \eqref{eq:separatrix}, see Figure~\ref{fig:phase_diagram} and Remark \ref{rem:weighted-fair-comp} above. Hence we are able to provide  a very explicit and exact classification of fitness advantage areas. However, we also see that $\vp_{\mathrm{res}}>\vp_{\mathrm{sto}}$ when environmental states \emph{both} change rarely, i.e.\ also if $s_1$ and $s_2$ are small. Since \cite{MS08} only consider environmental cycles of fixed length $T\equiv s_1^{-1}+s_2^{-1}$, they cannot observe this effect. In contrast, \cite{MS08} also provide results for the limit of extremely quickly fluctuating environments, which are meaningless in our model, since we assume discrete time/generations.

\begin{quote}
  \emph{(3)\hspace{1em} When the environment changes almost every generation, i.e. when $s_1\cdot s_2$ is large, the prescient switcher is the fittest.}
\end{quote}
The prescient switcher -- to the authors' knowledge -- does not appear in the biological literature and is not discussed in \cite{MS08}. A possible reason for this might be that, in comparison to the responsive strategy, the prescient switcher seems counter-intuitive. However, when the randomness of the environment is negligible and the environment becomes predictable, this strategy is dominant as Figure~\ref{fig:phase_diagram} demonstrates. Annual plants serve as a suitable example.

\begin{quote}
  \emph{(4)\hspace{1em} For intermediate values of $s_1\cdot s_2$, stochastic switchers emerge as dominant strategies.}
\end{quote}
As discussed in Remark~\ref{rem:convex-combination}, stochastic switchers can be constructed via convex combinations of fairly compared responsive and prescient switchers, retaining rank 1. Their Lyapunov exponents, $\vp_{\textnormal{cc}}(q(1),q(2))$, can then be computed by Theorem~\ref{thm:stochastic-fitness} and maximized in $q(1)$ and $q(2)$ to obtain the optimal rank-1-strategy under fair comparison, which is a non-trivial convex combination for intermediate $s_1\cdot s_2$ as Figure~\ref{fig:phase_diagram_cc} suggests.

However, in general the optimal strategy under fair comparison might not be of rank 1, as Figure~\ref{fig:non-zero-det} suggests that non-zero determinants might further increase fitness in certain scenarios.

Interestingly, we see in Figure~\ref{fig:phase_diagram_cc} that mixed strategies can be uniquely super-critical (``strong fitness advantage''). This observation has no analogue in \cite{MS08}, since although they mention `hybrid' strategies (p.~1144), they do not provide any results for them.  Note that mixed strategies could be seen as a kind of bet-hedging on the level of switching strategies, and we thus suggest the term `second-level bet hedging'. Classical bet-hedging, first developed in the context of plants \cite{Co66}, but common also in (isogenic) microbial populations \cite{JHK11}, which refers to keeping phenotypic variability, may then be considered as `first level bet hedging'.

\subsection*{The phenotypic switching model of Dombry, Mazza and Bansaye and the hereditary vs.\ non-hereditary case.}
Regarding the results in \cite{DMB11} on branching processes in random environment,  recall that the authors deal with a general framework for \emph{phenotypic switches} between potentially many types, and with a much more general class of random environments than in our model (though still assumed to be stationary and ergodic). However, in some regards (when considering active and dormant states as different phenotypes that one can switch between), their model and results are also more restrictive than ours, and some scenarios of dormancy-related reproduction are not covered. To understand these differences, let us recall their distinction between \emph{hereditary} and \emph{non-hereditary} reproduction resp.\ switching strategies. In non-hereditary strategies, in a first step, the offspring numbers of individuals are sampled, and then, in a second step, \emph{independently}, the new phenotypes are attached to the offspring individuals. This case clearly disentangles reproduction and phenotype-allocation and allows to obtain a wealth of elegant results on the fitness and optimality of switching strategies. In contrast, the hereditary case does not feature this disentanglement, and type allocation may depend on the type of the parents. Because of this, this case is mathematically much harder to investigate. In fact, here, \cite{DMB11} provide no systematic results for the Lyapunov exponents of the system (though they give a bound for the `finite time growth rate').

Unfortunately, the mathematically tractable non-hereditary case already excludes our simple example for dormancy-related reproduction being the result of either binary fission or sporulation from Section~\ref{sec:mot-exp}, since here, the offspring number determines the offspring type (two in the case of fission, one in the case of sporulation).  It is still possible to transfer some of their machinery to the cases which in our models correspond to rank-1 matrices, but not to the rank-2 case.

Nevertheless it is interesting to compare some of the results  for the non-hereditary case with our results, at least on a qualitative level. In the case with spontaneous switching (related to the `no-sensing' case in the language of \cite{DMB11}), they show that there are situations where a strategy that produces several phenotypes can have a strong fitness advantage over the simple `single-type' strategy (`sleepless' without dormancy resp.\ phenotypic variability), cf.\ their Section 1.1.1. For the case with `sensing' (which corresponds to responsive switching) and for $s_1=s_2$, they show that in the presence of rarely changing environments, responsive switching is optimal, whereas in highly fluctuating environments, prescient switching dominates. Furthermore, a mixed strategy -- i.e. stochastic switching -- is optimal in intermediate regimes. This corresponds to our observation in Figure~\ref{fig:phase_diagram} (left). However, similarly to the iid case given on the line where $s_1=1-s_2$, this does not provide a full picture of the fitness landscape.

\subsection*{The role of relative reproductive costs and `weighted fair comparison'}
Note that while \cite{DMB11} do consider mixed switching strategies (in contrast to \cite{MS08}), their modeling always implicitly implies a `fair-comparison' of reproductive strategies (cf.\ their modeling with fixed type distributions $\Upsilon_{t,e}$ on p.~377), corresponding to our comparison with $\ga = 1$. If this assumption is violated, that is,  dormant offspring are either `more cost efficient' than active offspring ($\ga < 1$), or `more expensive' ($\ga > 1$), the picture regarding optimal strategies becomes very rich and exhibits novel effects. Note that situations in which dormant offspring are more expensive than active offspring could relate for example to the sporulation process of Bacillus subtilis, which takes much longer than producing an active offspring by binary fission \cite{PH04}, and thus leads to fewer dormant offspring per time unit, resulting in higher `effective' costs. On the other hand, plants often  produce many seeds at a low cost, and this is clearly the optimal strategy in the presence of extremely harsh environments (`winter') that effectively kills all `active' individuals from the current generation of a species.

In the present paper, the following picture emerges (cf.\ Figure~\ref{fig:phase_diagram:fair}): Under reduced costs for dormant individuals (here, $\ga = 1/2$), rule (1) still holds in a qualitative sense; however, the region where the sleepless case is optimal is further reduced, while the relation between the regions of the switching strategies does not seem to change qualitatively.  This suggests the following rule:

\begin{quote}
  \emph{(5)\hspace{1em} Low relative reproductive costs for dormant individuals may strongly increase the effectiveness of all switching strategies, not necessarily changing the interrelation between switching strategies.}
\end{quote}
Under increased costs for dormant offspring (here, $\ga = 2$), an entirely new effect appears. Here, the `sleepless' population can suddenly have a strong fitness advantage in moderately fluctuating environments, at the cost of the stochastic switcher.

\begin{quote}
  \emph{(6)\hspace{1em} High relative reproductive costs for dormant individuals may strongly increase the effectiveness of the sleepless strategy if environments fluctuate quickly.}
\end{quote}
Interestingly, both the responsive and the prescient strategies seem less severely affected by variable relative reproductive costs as both retain their strong advantages in Figure~\ref{fig:phase_diagram:fair} (right).

\begin{quote}
  \emph{(7)\hspace{1em} The responsive and the prescient switching strategies are relatively robust under moderate changes in relative reproductive costs in its region of dominance.}
\end{quote}
However, for strong relative fitness differences, the qualitative picture may change drastically. For example, in Figure~\ref{fig:gamma} (left) the separatrix between the regions of dominance of the responsive and the stochastic switcher performs a phase transition, where for $\ga=1/9$ it becomes a straight line. This invites a more comprehensive study of the sensitivity of the optimal strategies on the relative reproductive cost, which however seems beyond the scope of the present paper.

\subsection*{Stochastic vs.\ deterministic modeling}
Given the long tradition of stochastic vs.\ deterministic modeling in population dynamics, here represented by branching processes vs.\ dynamical systems, it is interesting to assess, at least rudimentarily, which similarities and differences of the conclusions under the respective modeling frameworks can be attributed to these modeling assumptions.

A first observation is that several qualitative results remain valid under both modeling assumptions (e.g.\ regarding rule 1). The mathematical reason is that the maximal Lyapunov exponent in the stochastic case depends on the mean matrices of the offspring distributions of the branching processes and thus typically agrees with those of the dynamical systems. We can imagine that optimal population strategies in deterministic environments might often be deterministic as well. However, as Figure~\ref{fig:phase_diagram} suggests, a \emph{truly} random strategy such as the stochastic switcher can only prevail in a \emph{truly} random environment, i.e.\ when neither high nor low frequency of environmental changes allows for a deterministic approximation of the environment. Hence, especially when investigating stochastic switching strategies one should always take into account random environments.

Differences do appear for example when taking the limit in extremely quickly fluctuating populations (as in \cite{MS08}), which is meaningless in our discrete-time model. This could be remedied by switching to continuous-time birth-death processes instead of branching processes, but having these extremely quick fluctuations seems not very realistic in either case.

We would expect strong differences (and in fact major advantages of stochastic modeling) in situations when population size may fluctuate strongly, and in particular may be very small, so that stochasticity has a strong effect. This could for example be the case in scenarios when new dormancy traits invade a resident population without this trait (as in \cite{BT20}), or infections at an early stage, but this has not been considered in the present paper.

Stochasticity is certainly also relevant when considering extinction probabilities, which by definition involve small population sizes. Related questions have been treated by Jost and Wang \cite{JW14}, which investigate the extinction probability of branching processes under optimal phenotype allocation.

When population size is reasonably large, seminal results of Kurtz \cite{Ku71} (see also the comprehensive theory in \cite{EK86}) show that finite-type birth-death processes (and thus also discrete-time branching processes under suitable conditions) converge uniformly to the corresponding dynamical systems as the population size tends to infinity. This can be generalized to spatial and measure-valued set-ups, see the (also seminal) paper by Fournier and M\'el\'eard \cite{FM04}.

\subsection{Outlook}
Note that our results are still  incomplete. For example, we are only able to provide exact results for certain classes of switching regimes, which however include important special cases. The other cases could be approximated by the methods outlined in Section~\ref{sec:rank-2-case}, or by simulation. The results are often in line with intuition, but beyond that give exact quantitative insights. 

Our study thus invites further research in several directions, which we now briefly outline. We distinguish between more theoretical / mathematical and empirical / biological tasks and project ideas, some of which are admittedly speculative.\\

\emph{On the mathematical side,} progress regarding the exact computation of Lyapunov exponents is certainly desirable, but this is known to be difficult and probably needs particular and sophisticated methods for particular switching strategies, depending on the algebraic properties of the underlying mean matrices. Systematic and comprehensive progress is thus still elusive.

A promising and more readily accessible task is to extend the modeling frame of our BGWPDRE. For example, one could involve much more general environmental processes (still stationary and ergodic), such as in \cite{DMB11}. Further, one might wish to switch from discrete time/generation branching processes to continuous-time birth-death processes. This also invites to switch to an `adaptive dynamics' related set-up, in which one could try to merge random environments and dormancy with competition and mutation. A starting point could be the recent model  on dormancy under competition in \cite{BT20}, extended by rates depending on the state of a (deterministically or  randomly fluctuating) environment. Having more than one species with potentially different dormancy strategies, perhaps even in a spatial set-up \`a la \cite{FM04}, could lead to truly ecological models. If necessary, these could be approached by simulation instead of rigorous analysis. Highly interesting in this context would also be to replace the `random' environment by treatment protocols to optimize e.g.\ the efficiency of anti-biotic treatment.\\

\emph{On the biological side}, it seems necessary to calibrate the above systems (and the more sophisticated models yet to be developed) to the behaviour of model species, where switching strategies  are known and parameters could be estimated. This invites experiments under controlled lab conditions, where, e.g., environmental conditions could be externally controlled. It would be very interesting to see in how far theoretically predicted patterns manifest themselves in these experiments. 

This goes hand-in-hand with theoretical assessments  in clinical set-ups. We think that inter-disciplinary cooperation could be rewarding here, and the recent work on stochastic individual based modeling of a certain immune-therapy of cancer \cite{G+19} could be seen as a promising example.

Finally, a very interesting field of research, combining mathematical and biological aspects, could be related to the effects of long-term changes in the distribution of the random environment (for example due to climate change). It is well-known that climate change can have a serious impact on seed banks, see e.g.\ \cite{Oo12}. From a mathematical point of view, it would be interesting to understand the robustness of switching methods, or even bet-hedging strategies on the level of switching strategies (by using mixtures of switching strategies), under changing environmental distributions. It seems important to understand and predict the presence of `trigger points', when the underlying systems might completely change their behaviour (e.g.\ from super-critical to sub-critical). Of course, mathematically this means that one would have to move away from the stationarity and ergodicity assumption of the random environment, which will pose significant challenges. A first way to approximate such scenarios could be to investigate environments whose distribution involves two different time-scales (preserving the stationarity assumption), and where the second time-scale is much longer than the first. Rescaling of the system may then lead to a separation of time-scales which is tractable, again with potentially interesting mathematical and ecological implications.

\section{Appendix}
\label{sec:proof-mot-exp}
In this appendix, we provide the proof of Proposition \ref{prop:motivating-exp} for reference.
\begin{proof}[Proof of Proposition~\ref{prop:motivating-exp}]
  For the classical BGW process $(X_n)$ with offspring distribution $Q_X$ we write $h\!: [0,1] \to [0,1]$, $s \mapsto h(s)=\mean[s^{X_1} \,|\, X_0 = 1]$ to denote the corresponding offspring probability generating function.  Set $h'(1)=\mu_X$ and $h(0)=Q_X(0)$.  Recall that, by assumption, the variance of $Q_X$ is finite.  It is well-known, cf.\ \cite[Theorem I.5.1]{AN72}, that the survival probability $\si_X$ is given by $1-x^\ast$, where $x^\ast$ is the smallest fixed point of $h_X$.  Furthermore it is known that $(\mu_X)^{-n}\prob\!\big[X_n>0\big]$ for $n\to\infty$ converges to a positive limit if $\mu_X<1$, as does $n\prob\!\big[X_n>0\big]$ if $\mu_X=1$ (cf. \cite[Corollary I.11.1]{AN72} and \cite[Theorem I.9.1]{AN72} respectively).  Since $X_n>0$ iff $T_{X}>n$, it follows that
  \begin{align*}
    \mean\!\big[T_{X}\big]
    \;=\;
    \sum_{n \geq 0}\prob\!\big[T_{X}>n\big]
    \;=\;
    \sum_{n\geq0}\prob\!\big[X_n>0\big]
    % \begin{cases}
    %   = \infty & \text{if }p=\tfrac12,\\[.5em]
    %   < \infty & \text{if }p<\tfrac12.
    % \end{cases}
  \end{align*}
  is infinite if $\mu_X=1$ and finite if $\mu_X<1$.

  Coming to $(Z_n)$, its offspring probability generating function is given by $g:[0,1]^2\to[0,1]^2$ with
  \begin{align*}
    g(s_1,s_2)
    &\;\ldef\;
      \begin{pmatrix}
        \mean\Big[s_1^{Z_1^1}s_2^{Z_1^2} \,\Big|\, Z_0=(1,0)\Big] \\[.5em]
        \mean\Big[s_1^{Z_1^1}s_2^{Z_1^2} \,\Big|\, Z_0=(0,1)\Big]
      \end{pmatrix}.
  \end{align*}
  Now, by \cite[Theorem V.3.2]{AN72}, $1 - \si_Z$ can be given as the first component of the smallest fixed point of $g$. Denoting $w = Q_Z^2(1,0)$ and $d = Q_Z^2(0,0)$ as in the example in Section~\ref{sec:mot-exp}, basic tranformations yield that $g(s_1,s_2) = (s_1,s_2)$ iff
  \begin{align*}
    \begin{pmatrix}
      g_1(s_1,s_2)\\
      s_2
    \end{pmatrix}
    &\;=\;
      \begin{pmatrix}
        s_1 \\
        \frac{d+ws_1}{d+w}
      \end{pmatrix},
  \end{align*}
  such that by monotonicity $1-\si_Z = g_1(1-\si_Z,\frac{d+w(1-\si_Z)}{d+w})$.
  From \eqref{eq:fair-comp-mot-exp} we obtain for $s\in[0,1]$ that
  \begin{align*}
    h(s)
    \;=\;
    1-\sum_{k\geq1}Q_X(k)(1-s^k)
    \;\leq\;
    g_1(s,s)
    \;<\;
    g_1(s,\tfrac{d+ws}{d+w}),
  \end{align*}
  which implies that $1-\si_Z \geq 1-\si_X$, i.e.\ $\si_Z \leq \si_X$, where the inequality is strict, if $\si_X > 0$.

  It remains to prove the results regarding $T_Z$.  For this, denoting by $m_1 = \mean[Z_1^1 \mid Z_0=(1,0)]$ and by $m_2 = \mean[Z_1^2 \mid Z_0=(1,0)]$, the mean matrix of the offspring distribution of $Z$ is given by
  \begin{align}
    M
    \;=\;
    \begin{pmatrix}
      m_1 & m_2 \\ 
      w & 1-w-d
    \end{pmatrix}\notag
    % \label{eq:expectation-matrix}
  \end{align}
  along with its largest eigenvalue
  \begin{align}
    \vr
    \;=\;
    \vr(w,d)
    &\;\ldef\;
    \frac{1}{2}\,
    \Big(m_1 + 1-w-d + \sqrt{(m_1-(1-w-d))^2 + 4wm_2}\Big)
    \nonumber\\[0.5ex]
    &\;\geq\;
    \frac{1}{2}\, \Big(m_1 + (1-w-d) + \big|m_1 - (1-w-d)\big|\Big)
    \nonumber\\[1ex]
    &\;=\;
    \max\{m_1,1-w-d\}.
    \label{eq:rho-lower-bound}
  \end{align}
  Then, $(Z_n)$ survives with positive probability iff $\vr > 1$, while $\prob\!\big[|Z_n|>0\big]\xrightarrow{n\to\infty}0$, if $\vr \leq 1$ (cf.\ \cite[Theorem V.3.2]{AN72}).
  More importantly, for case \emph{(3)}, when $d$ is so small that $1-d > \mu_{Z,1}$, \eqref{eq:rho-lower-bound} implies $\vr(0,d) > \mu_{Z,1}$.  Thus, by continuity, there also is $w > 0$ such that $\vr(w,d) > \mu_{Z,1}$.
  
  We now apply \cite[Theorem~V.4.4]{AN72}: Since the offspring distributions are of finite variance, the second moment condition holds and the theorem implies
  \begin{align*}
    \lim_{n\to\infty} \vr^{-n} \prob\!\big[Z^1_n+Z^2_n>0\big]
    \;\in\;
    (0,\infty),
  \end{align*}
  which concludes the proof of \emph{(3)}.

  For the critical case \emph{(2)} note that $\vr \leq \max\{m_1+m_2,1-d\}$. (The maximum row sum can be seen as an operator norm and hence is an upper bound for all eigenvalues.)  Hence, in the case $m_1 + m_2 = \mu_{Z,1} < \mu_X = 1$ the proof is complete, since $\prob[T_Z > n] \approx \vr^n$.  At last, it remains to show that even if $m_1 + m_2 = 1$, $\vr < 1$.  Thus, letting $m_1 = 1-m_2$ note that $\vr = 1$ iff
  \begin{align*}
    \sqrt{(w+d-m_2)^2+4wm_2}
    \;=\;
    m_2 + w+d,
  \end{align*}
  which can only hold if either $d=0$ or $m_2=0$, both contradicting our assumptions.
\end{proof}

\textbf{Acknowledgements:} The authors gratefully acknowledge support by DFG SPP 1590 ``Probabilistic structures in evolution'', project BL 1105/5-1.
The authors also thank Matthias Hammer for valuable comments and discussions.
\\

\bibliographystyle{abbrv}
\bibliography{literature}

\end{document}